\documentclass{article}

\usepackage[utf8]{inputenc}
\usepackage[centertags]{amsmath}
\usepackage{amsfonts}
\usepackage{amssymb}
\usepackage{amsthm}
\usepackage{newlfont}
\usepackage{mathtools}
\usepackage[top=1in, bottom=1in, left=1in, right=1in]{geometry}
\usepackage{bbold} 
\usepackage{tikz}
\usepackage{dsfont}
\usepackage{caption}
\usepackage{subcaption}
\usepackage{tikz}
\usepackage{tkz-graph}
\usepackage{float}  
\usepackage{booktabs}

\usepackage{hyperref}

\theoremstyle{plain}
\newtheorem{theorem}{Theorem}[section]
\newtheorem{proposition}[theorem]{Proposition}
\newtheorem{corollary}[theorem]{Corollary}
\newtheorem{lemma}[theorem]{Lemma}

\newtheorem{question}{Question}

\theoremstyle{definition}
\newtheorem{definition}[theorem]{Definition}

\newtheorem{condition}[theorem]{Condition}
\newtheorem{observation}[theorem]{Observation}

\newtheorem{example}{Example}[section]

\newcommand{\Hidden}[1]{}

\newcommand{\N}{\mathbb{N}}

\DeclarePairedDelimiterXPP\tr[1]{\operatorname{tr}}{(}{)}{}{#1}

\newcommand{\kemeny}{\mathcal{K}}

\newcommand{\1}{\mathbb 1}

\DeclareMathOperator{\vol}{vol}

\renewcommand{\P}{\mathbb{P}}
\newcommand{\E}{\mathbb{E}}

\DeclarePairedDelimiter{\floor}{\lfloor}{\rfloor}

\title{
    On defining Kemeny's constant for non-backtracking random walks
}
\author{
    Jane Breen\thanks{Faculty of Science, Ontario Tech University, Oshawa, ON L1K1J9, Canada (\texttt{jane.breen@ontariotechu.ca}).
J. Breen is supported by the Natural Sciences and Engineering Research Council of Canada (NSERC) Grant RGPIN-2021-03775.},
    Mark Kempton\thanks{
        Department of Mathematics, Brigham Young University, Provo UT 84602, U.S.A. (\texttt{mkempton@mathematics.byu.edu}).
    }, 
    Adam Knudson\thanks{
    Department of Mathematics, University of California Los Angeles, Los Angeles, CA 90095, U.S.A. (\texttt{adamknudson@math.ucla.edu}).},
    and 
    Matthew Shumway\thanks{
        Department of Mathematics, Brigham Young University, Provo UT 84602, U.S.A.
        (\texttt{mwshumway7@gmail.com}).}
}
 \date{}

\begin{document}


\maketitle
\begin{abstract}
    We propose two possible definitions for a version of Kemeny's constant of a graph based on non-backtracking random walks (in place of the usual simple random walk).  We show that these two definitions coincide for edge-transitive graphs, and give a condition generalizing edge-transitive for which equality holds, and investigate by how much they can differ in general.  We compute our non-backtracking Kemeny's constant for several families of graphs.  
\end{abstract}

\bigskip



\section{Introduction}

Kemeny's constant of a Markov chain is an important parameter measuring how quickly on average the Markov chain can travel from one state to another.  We will be studying Kemeny's constant for random walks on simple undirected graphs.  This constant is a graph parameter giving an indication of how well connected the graph is and how quickly the random walk can move among the vertices of the graph.  

Considerable research has been done in recent years comparing non-backtracking random walks to the usual simple random walks on graphs. In many contexts, the non-backtracking condition is intuitive and natural as a model for processes on a network. Early work suggests that in most cases, the non-backtracking condition leads to a faster mixing rate (see \cite{alon2007non,kempton2016non}). Non-backtracking variants of PageRank \cite{aleja2019non,arrigo2019non,glover2022effects} and graph clustering algorithms \cite{krzakala2013spectral} have been studied, where it is hoped that the non-backtracking condition leads to improved performance.  

In this work, we study how Kemeny's constant is affected when we consider non-backtracking random walks in place of simple random walks on a graph.  The study of a non-backtracking version of Kemeny's constant was initiated by the authors in \cite{breen2023kemeny}.  In that work, we considered Markov chains on the state space of directed edges of the graph (where each edge in our simple undirected graph is viewed as two directed edges).  This has the advantage that, viewed on this state space, the non-backtracking random walk is a Markov chain.  However, many questions about Kemeny's constant that are frequently studied ask about comparisons among graphs with the same number of vertices, or families of graphs based on the number of vertices (see \cite{breen2019computing} for instance).  When working with the directed edges as the state space, such questions can be obscured by differing numbers of edges in the graph.  Thus it is desirable to have a version of the non-backtracking Kemeny's constant that still uses the vertices as the state space.  The difficulty here is that on the state space of vertices, the non-backtracking walk is not a Markov chain.  Overcoming this difficulty is the main subject of this paper.  

In this paper we propose two possible definitions for a ``vertex space" version of the non-backtracking Kemeny's constant---one based on non-backtracking mean first passage times, and one based on a non-backtracking version of the fundamental matrix of the random walk.  Both are based on strong analogies with formulas for the usual Kemeny's constant of a graph.  Our main result is to show that these two definitions coincide for edge-transitive graphs (and state a more general condition that also implies they are the same).  We also investigate how far these two definitions can be from one another in general, and give several examples.


Our work draws heavily on recent work from \cite{fasino2023hitting} on second-order random walks (i.e. random walks with two steps of ``memory").  Here, mean first passage times were defined between vertices of the graph by ``projecting" the edge-space quantities onto the vertices. We want to build on this by performing comparable projections on Kemeny's constant in this context of second-order random walks.


For the interested reader, we have made available at a \href{https://github.com/mwshumway/NBRW/tree/main?tab=readme-ov-file}{GitHub page} \cite{githubpage}
several Sage functions that we have relied on for numerical computations of many of the quantities involved in our study. This code is one of the additional contributions of our work.

\section{Mathematical preliminaries}

\subsection{Markov chains and Kemeny's constant}\label{sec:kemeny_prelim}
A discrete, finite-state, time-homogeneous Markov chain can be thought of as a stochastic process on a finite state space, say $\mathcal{S} = \{s_1, s_2, \ldots, s_n\}$. Letting $X_k$ be the random variable denoting the state of the process at time-step $k$ (i.e.~$X_k\in \mathcal{S}$), the sequence $\{X_k \colon k\in\mathbb{N}\}$ is a Markov chain if it has the property that 
\[\P(X_{k+1}= x_{k+1} \mid X_k = x_k, \ldots, X_0 = x_0) = \P(X_{k+1} = x_{k+1} \mid X_k = x_k).\]
This can be interpreted as saying that the probability of transitioning to a given state in the next time step depends only on the current state of the system. That is, the probability $p_{ij}$ of transitioning from state $s_i$ to state $s_j$ in a single time-step is fixed. Thus any Markov chain is defined completely by its probability transition matrix $P=[p_{ij}]$, which is a non-negative row-stochastic matrix,~i.e. its rows sum to 1, or $P\1 = \1$, where $\1$ represents the column vector of all ones. Furthermore, the long-term behavior of the system is governed by powers of this matrix, in that the $(i,j)$ entry of $P^k$ is the probability of the process being in state $s_j$ at time $k$, given that it started in state $s_i$. If the matrix $P$ is \emph{primitive}, then, due to the Perron-Frobenius theorem, \[\lim_{k\to\infty} P^k = \mathbb{1}\pi^\top,\] where $\pi$ is a left eigenvector of $P$ corresponding to the Perron eigenvalue 1, normalized so that the entries of $\pi$ sum to 1. Thus $\pi$ represents a probability distribution, and this result may be interpreted to say that the long-term probability of being in the $j^{th}$ state at time $k$ is equal to $\pi_j$ as $k$ tends to infinity, independent of the initial state of the process. This vector $\pi$ is referred to as the stationary distribution, and the $j^{th}$ entry may be interpreted as a measure of the relative `importance' of $s_j$ in the Markov chain.

To quantify the short-term behavior of the Markov chain, one can discuss the first passage times, or first hitting times. Let $F_{ij}$ be the random variable denoting the first time $s_j$ is reached, given that the process starts in $s_i$. Then the expected value of this random variable $\E(F_{ij})$ is referred to as the \emph{mean first passage time} from $i$ to $j$, and denoted $m_{ij}$. (Note: as convention, the first passage time from $i$ to $i$ is considered to be the first time $s_i$ is reached after leaving $s_i$, or the \emph{first return time}. The first hitting time $F_{ii}$ is considered to be zero. Otherwise, hitting times and first passage times are equivalent.) It is not difficult to argue by conditioning on the first step of the chain that the mean first passage times satisfy the following equation:
\[m_{ij} = \sum_{\substack{k=1\\k\neq j}}^n p_{ik}m_{kj} + 1. \]
Thus it follows (see \cite[Section 6.1]{stevesbook}) that the matrix of mean first passage times is the unique matrix $M= [m_{ij}]$ that satisfies the matrix equation
\[M = P(M-M_{dg}) + J,\]
where $M_{dg}$ denotes the diagonal matrix with entries $m_{11}, m_{22}, \ldots, m_{nn}$, and $J$ is the all-ones matrix. It is easily shown from here that the mean first return time $m_{ii} = \frac{1}{\pi_i}$; i.e. $M_{dg}$ can be found to be $\Pi^{-1}$, where $\Pi=\mathrm{diag}(\pi)$. 

Rather than rely on probability arguments to find the mean first passage times, one can solve the above matrix equation using \emph{generalized inverses}. In particular, $M$ is the solution to \[(I-P)M = J-\Pi^{-1},\] 
and while $(I-P)$ is a singular matrix, there is nevertheless a unique solution to this equation.
In \cite{kemenysnell}, the \emph{fundamental matrix} of an ergodic Markov chain is introduced, defined as 
\begin{equation}\label{eq:fund_matrix}
    Z = (I-P+\1\pi^\top)^{-1}.
\end{equation}
This serves as a generalized inverse for $I-P$, with the properties that $ZP = PZ$, $Z\1 = \1$, $\pi^\top Z = \pi^\top$, and $(I-P)Z = I - \1\pi^\top$. This matrix is referred to as the fundamental matrix of the chain as it is ``the basic quantity used to compute most of the interesting descriptive quantities for the behavior of the Markov chain'' (see \cite[Section 4.3]{kemenysnell}). In particular, the matrix $M$ of first hitting times (i.e.~using the convention that $m_{ii}=0$) can be written as 
\[M = (-Z+JZ_{dg})\Pi^{-1}.\]  

We are now ready to introduce Kemeny's constant, and discuss various formulations of it which will be explored throughout this paper.

Kemeny's constant was first introduced in 1960 in \cite{kemenysnell} as follows: Fix a starting state $s_i$, and consider the quantity $\kappa_i = \sum_{j\neq i}m_{ij}\pi_j.$ This can be interpreted as the expected length of trip from state $i$ to a randomly-chosen state (with respect to the probability distribution $\pi$). Astonishingly, this was found to be independent of the choice of index $i$! This is shown in a rather straightforward manner by considering that 
\begin{eqnarray*}
M\pi &=& (-Z+JZ_{dg})\Pi^{-1}\pi\\
& = & (-Z+JZ_{dg})\1\\
& = & - \1 + JZ_{dg}\1\\
& = & (\mathrm{trace}(Z)-1)\1.
\end{eqnarray*}
As such, it follows that the quantity $\kappa_i$ defined above is equal to trace$(Z)-1$ for all $i$. It is denoted $\mathcal{K}(P)$, without dependence on $i$, and was later referred to as \emph{Kemeny's constant} (see \cite{grinstead1997introduction}). Since $\pi^\top\1 = 1$, it follows that we can also write $\mathcal{K}(P)$ as $\pi^\top M\pi$, which gives the expression 
\[\mathcal{K}(P) = \sum_{i=1}^n\sum_{\substack{j=1\\j\neq i}}^n \pi_im_{ij}\pi_j,\]
and can be interpreted as the expected length of a random trip between states in the Markov chain, where both the starting and ending states are chosen randomly with respect to the stationary distribution of the chain. For this reason, Kemeny's constant has become a useful quantifier of the behavior of a Markov chain, described as an \emph{expected time to mixing} and relating to how well-conditioned the Markov chain is (see \cite{hunter2006mixing}).  

In 2002 in \cite{levene2002kemeny}, Levene and Loizou observed that based on the expression for $\mathcal{K}(P)$ in terms of the trace of the fundamental matrix, there is a simple expression for it in terms of the eigenvalues of $P$. In particular, if the eigenvalues of $P$ are $1, \lambda_2, \ldots, \lambda_n$, then 
\[\mathcal{K}(P) = \sum_{j=2}^n \frac{1}{1-\lambda_j}.\]
This expression in particular has sparked a great deal of interest in the value of Kemeny's constant as a spectral graph invariant in the case that we consider random walks on graphs.


\subsection{Random walks and non-backtracking random walks}\label{sec:nbprelim}
A graph $G$ is defined in terms of a vertex set $V(G)$ and an edge set $E(G)$. In this work, we consider simple, undirected graphs, in which the edge set consists of unordered pairs of vertices $\{u, v\}$, indicating that the vertices $u$ and $v$ are joined by an edge. In this instance, the vertices $u$ and $v$ are said to be \emph{adjacent}, $u$ and $v$ are both said to be \emph{incident} to the edge $\{u, v\}$, and $u$ is a neighbor of $v$ (and vice versa). The number of neighbors of a vertex $v$ is called the \emph{degree} of $v$, and is denoted by $\deg(v)$ or $d_v$.

A random walk on a graph $G$ is the Markov chain whose states are vertices of the graph, and the transition from one state to the next chooses uniformly at random from among the neighbors of the current vertex.  As such, the transition matrix is given by $P=D^{-1}A$ where $A$ is the adjacency matrix of $G$, and $D$ is the diagonal matrix of degrees of each vertex.  It is well-known that the stationary distribution $\pi$ for this Markov chain is given by $\pi_j=\deg(j)/2m$ where $m$ is the number of edges of $G$.

A \emph{non-backtracking random walk} on a graph $G$ is one in which the random walker is not permitted to immediately return to a vertex visited in a previous step. That is, if the random walker visits $u$ and then its neighbor $v$, the next step must choose among the neighbors of $v$ other than $u$. We will use the abbreviations NBRW to refer to a non-backtracking random walk, and SRW to refer to the usual simple random walk.  Since the transition probabilities depend not only on the current state of the process but also the previous state, this is no longer a Markov chain. However, one can equivalently think of the random walker as occupying the directed edges of the graph, and define a Markov chain which accounts for the previous two vertices visited. More generally this is called a second-order Markov chain (see \cite{fasino2023hitting}), and is defined with a state space consisting of ordered pairs of the original states. Such ordered pairs $(u,v)$, in the context of graphs, can be considered as \emph{arcs} or \emph{directed edges} from $u$ to $v$. If the random walker's state is this edge $(u,v)$, then the state on the next step is some edge $(v,w)$.  The non-backtracking condition simply says that we choose uniformly at random from among edges of the form $(v,x)$ for $x$ other than $u$.  Thus the possible next states depend only on the current state (the directed edge $(u,v)$).  Thus the non-backtracking random walk on this state space of directed edges is a Markov chain.  See \cite{breen2023kemeny,fasino2023hitting, glover2021some, kempton2016non} for discussions and examples of this perspective.  Throughout this paper, we will talk about taking random walks (simple or non-backtracking) in the ``vertex space" or the ``edge space" respectively. 

We will denote by $P_{nb}$ the transition matrix for the NBRW on the edge space.  That is, $P_{nb}$ is the $2m\times 2m$ matrix whose entries are given by \[P_{nb}((i,j),(k,\ell)) = \begin{cases}
    \frac{1}{\deg(j)-1}, &\text{ if } j=k \text{ and } \ell\neq i;\\
    0, &\text{ otherwise}.
\end{cases}\]  We will also find it convenient to talk about the matrices that transition back and forth between the edge space and the vertex space.  Define the $2m \times n$ matrix $S$ to be the matrix with entries \[S((i,j),x)=\begin{cases}
    1, & \text{ if } j=x;\\0, & \text{ otherwise.}
\end{cases}\]
Thus $S$ is the matrix for the operator that ``lifts" from the vertex space to the edge space.  Likewise, define the $n\times 2m$ matrix $T$ whose entries are given by \[T(x,(i,j))=\begin{cases}
    1, &\text{ if } i=x;\\0, &\text{ otherwise.}
\end{cases}\]  Thus $T$ is the matrix for the operator that ``projects" from the edge space to the vertex space.  We will frequently use these matrices to translate between vertex and edge space.  We note that an easy computation shows that $TS=A$ where $A$ is the adjacency matrix of the graph \cite{kempton2016non}.

It can be observed that the stationary distribution for the edge space non-backtracking walk, which we will call $\pi_e$ is the uniform distribution on the directed edges, thus $\pi_e = \frac{1}{2m}\1$. This is the same stationary distribution as for the simple random walk on the edge space.  Projecting this distribution from the edge space to the vertex space yields the usual stationary distribution $\pi$ where $\pi_j=\deg(j)/2m$.  Thus the non-backtracking random walk on an undirected graph has the same stationary distribution as the simple random walk \cite{kempton2016non}.

As we saw in Section \ref{sec:kemeny_prelim}, hitting times (or first passage times) for random walks are central to defining Kemeny's constant.  We will thus review how to compute hitting times for non-backtracking random walks. For a thorough treatment of this topic, see the recent work in \cite{fasino2023hitting}. 
\begin{definition}\label{def: nbrw first hitting time matrix}
Let $G$ be a graph, and consider a non-backtracking random walk $\{\Tilde{X}_k: k \in \N\}$ on the vertices of $G$; that is, where $\Tilde{X}_k \in V(G)$ for each $k$. Define $\Tilde{F}_{ij}$ to be the random variable counting the number of steps $k$ until $\Tilde{X}_k = j$, given that $\Tilde{X}_0 = i$. Then the non-backtracking first hitting time is $\mathbb{E}(\Tilde{F}_{ij})$, and the matrix of non-backtracking first hitting times is denoted $M_v^{(nb)}$, defined entrywise as

    \[
    \left(M^{(nb)}_v\right)_{ij} = \begin{cases}
        \mathbb{E}[\Tilde{T}_j \mid \Tilde{X}_0 = i], & \text{if } j \neq i; \\
        0,                            & \text{otherwise}.
    \end{cases}
    \]
\end{definition}

We now discuss the computation of this matrix.
Note that in the edge space of a graph, the random walk (either simple or non-backtracking) is considered to have ``arrived'' at a vertex $v$ if it is currently in any state (edge) of the form $(u,v)$. As such, the first hitting time from an edge $(i, j)$ to a vertex $v$ can be calculated using expressions from Markov chain theory for the hitting time from a state to a set of states, by taking a principal submatrix of the transition matrix and calculating a matrix inverse. In particular, for a Markov chain with transition matrix $P$, the vector of first hitting times to a set of states $\mathcal{S}^\ast$ is given by the row sums of the matrix $(I-P_{(\mathcal{S}^\ast)})^{-1}$. We denote by \( M^{(nb)}_{ev} \) the \( 2m \times n \) matrix with entries 
\[
(M^{(nb)}_{ev})_{(i, j), v} = \mathbb{E}[\widetilde{T}_v \mid \Tilde{X}_0 = i, \Tilde{X}_1 = j],
\]
representing this expected hitting time of vertex \( v \), conditioned on the first two steps of the non-backtracking random walk $\{\Tilde{X}_k : k \in \mathbb{N}\}$. 

With the matrix $M_{ev}^{(nb)}$ in hand, the $n\times n$ matrix $M^{(nb)}_{v}$ can then be computed using \begin{equation} \label{eq:Mnb_v_related_to_Mev}
    M^{(nb)}_{v} = D^{-1} T M^{(nb)}_{ev}.
\end{equation}




For more details on properties of these matrices, see \cite[Section 4]{fasino2023hitting} and Section \ref{subsec: 32} of this work.

\section{Kemeny's Constant: Vertex Space Perspectives}

\subsection{Definitional Approaches: First Hitting Times and a Fundamental Matrix}
Since the non-backtracking random walk does not form a Markov chain when the state space is restricted to the vertices of the graph, extending Kemeny’s constant to this context is not straightforward. In this section, we propose several candidate definitions for Kemeny’s constant in the case of the non-backtracking random walk on the vertices of a graph $G$. We will denote this generalized quantity by $\kemeny^{(nb)}_v(G)$.

\subsubsection{First Hitting Times}

Adopting the convention for first hitting times where \( m_{ii} = 0 \), recall that Kemeny’s constant for the simple random walk can be expressed as:  
\[
\kemeny(G) = \sum_{i=1}^n \sum_{j=1}^n \pi_i m_{ij} \pi_j = \pi^T M \pi,
\]  
where \( \pi \) is the stationary distribution vector, and \( M \) is the matrix of first hitting times between vertices.

Perhaps the most natural definition of $\kemeny^{(nb)}_v(G)$ arises from the above expression, where the entries of \( M \) are replaced with the non-backtracking first hitting times between vertices which we described how to compute in Section \ref{sec:nbprelim}.  We use this to propose our first definition of the non-backtracking Kemeny’s constant for a graph \(G\) as follows.

\begin{definition}\label{def: first hitting time kemeny}
    We define the \textit{first hitting time} Kemeny’s constant for a NBRW as \[
    \kemeny^{(nb)}_v(G) = \pi^T M^{(nb)}_v \pi,
    \]
\end{definition}

\noindent
It is worth noting that the above quantity is defined in terms of \(\pi\): it was first shown in Theorem 2 of \cite{kempton2016non} that the stationary distributions of the NBRW and the SRW are the same, which is $\pi_i = \frac{\deg(i)}{\vol(G)}$. 

Additionally, we clarify that the motivation for adopting the quadratic form arises from a key distinction: unlike its simple counterpart, the relationship $M^{(nb)}_v\pi = c\1$ does not necessarily hold. As characterized in Corollary 5.7 of \cite{fasino2023hitting}, the condition required for $M^{(nb)}_v\pi = c\1$ is that for all $i$, $\mathbb{E}[\Tilde{T}_i^+ \mid \Tilde{X}_1 = j, \Tilde{X}_0=i]$ is independent of $j$. Said differently, the return time of the non-backtracking random walk to vertex $i$ does not depend on the first step.

This condition will resurface many times, so we state it here for future reference.

\begin{condition}\label{cond: locally uniform return time}
    Given a graph $G$, we say that the non-backtracking random walk on $G$ has \emph{locally uniform return time}, if for all $i$, \[
    \mathbb{E}[\Tilde{T}_i^+ \mid \Tilde{X}_1 = j, \Tilde{X}_0=i]
    \]
    is independent of $j$.
\end{condition}

A more complete understanding of this condition is elusive, but desirable. One sufficient condition, given by the authors of \cite{fasino2023hitting}, is that for any two edges $e, f$ in the outgoing edges of vertex $i$, there exists an automorphism mapping $e \mapsto f$. This is satisfied for any edge-transitive graph, but is not a necessary condition as the following examples indicate.

\begin{example}
    Consider the two graphs on eight vertices shown below. Neither graph is edge-transitive, yet both satisfy Condition \ref{cond: locally uniform return time}. Interestingly, their analogues on a different number of vertices fail to do so (with the exception of the analogue of the left graph on six vertices, as this graph is edge-transitive). Empirical searches reveal that these are the only connected graphs on $n\leq 9$ vertices that meet this condition without being edge-transitive.
    \begin{figure}[ht]
        \centering
        \includegraphics[scale=0.5]{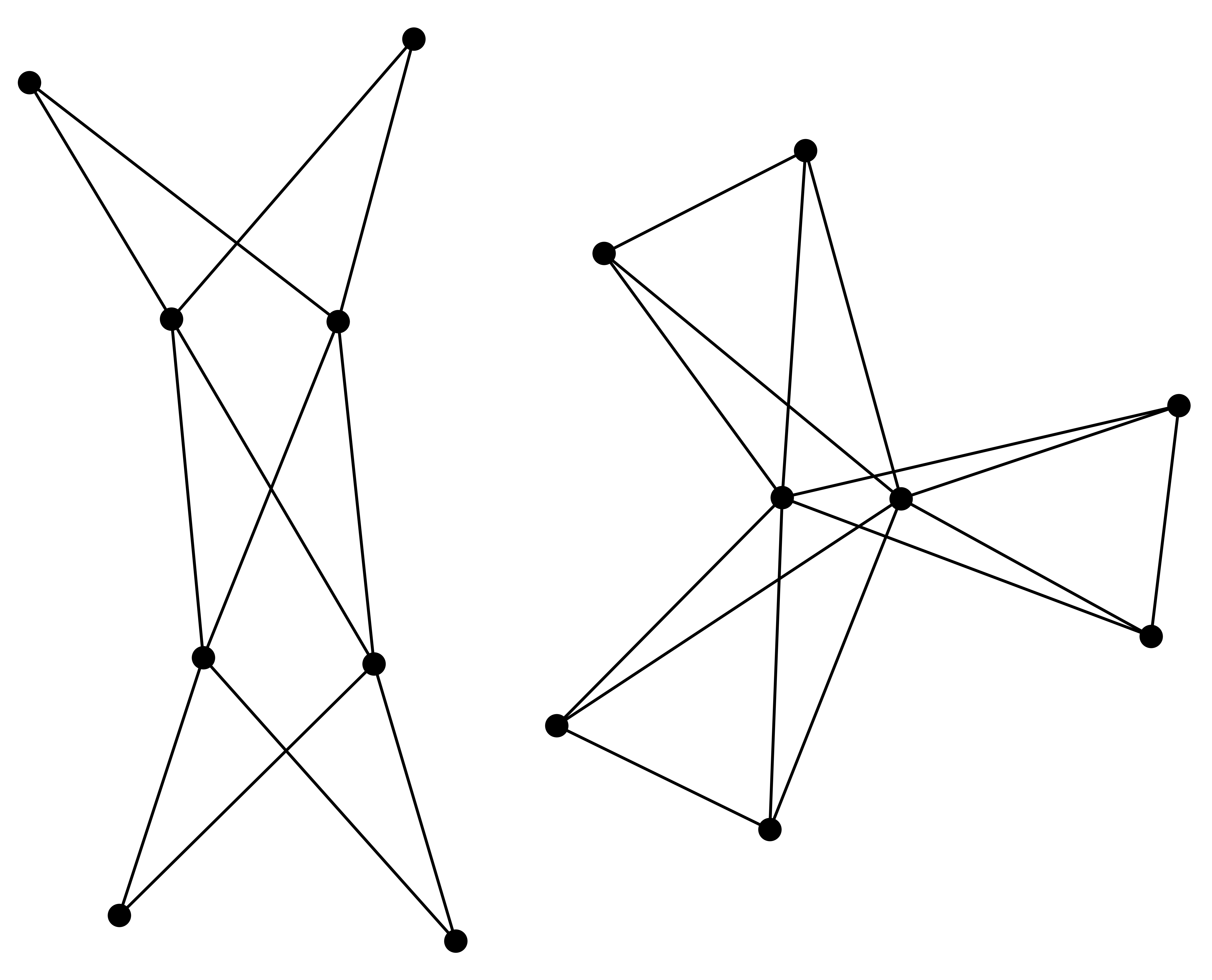}
        \caption{Two graphs that satisfy Condition \ref{cond: locally uniform return time} but are not edge-transitive.}
        \label{fig:condition_not_edge_transitive}
    \end{figure}
\end{example}

\subsubsection{Fundamental Matrix Projection}
Another possible avenue for defining Kemeny's constant in this context is through the trace of a fundamental matrix. For the simple random walk, the fundamental matrix $Z$ can be used to compute Kemeny's constant through the relation $\kemeny(G) = \text{tr}(Z) - 1$. The fundamental matrix exists for any ergodic Markov chain. 

For the non-backtracking random walk on the edge space, the fundamental matrix is given by  
\[
Z^{(nb)}_e = (I - P_{nb} + \mathbb{1}\pi_e^T)^{-1},
\]  
as in \eqref{eq:fund_matrix} and can be used to compute $\kemeny^{(nb)}_e(G)$ similarly to the SRW. While an explicit formulation of the non-backtracking fundamental matrix in the vertex space is not readily apparent, it is of interest to investigate whether an analogous construct with similar properties can be established. To pursue this, we begin with the following definition.

\begin{definition}
Let $G$ be a graph, and consider the non-backtracking random walk on the vertices of $G$, $\{\Tilde{X}_k: k \in \mathbb{N}\}$. Define an $n\times n$ matrix $\Tilde{P}^{(k)}$ whose $(i,j)$ entry is 
\[\Tilde{P}_{ij}^{(k)} = \P(\Tilde{X}_k = j \mid \Tilde{X}_0 = i).\]
That is, $\Tilde{P}_{ij}^{(k)}$ represents the probability of transitioning from vertex $i$ to vertex $j$ in $k$ non-backtracking steps.

\end{definition}

Next, we establish a relationship between this $k$-step transition probability matrix for the NBRW on the vertices, and the transition matrix for the NBRW on the directed edges of the graph. This relation is made clear in the following proposition.
\begin{proposition}\label{prop: vertex and edge transition prob relation}
    Let $k \in \mathbb{N}$. Then,
    \[
    \Tilde{P}^{(k+1)} = D^{-1}TP_{nb}^kS = D^{-1}TP_{nb}^{k+1}T^T,
    \]
\end{proposition}

\begin{proof}
    Let $\{W_k: k \in \mathbb{N}\}$ be a non-backtracking random walk on the directed edges of the graph, and let $\{X_k: k \in \mathbb{N}\}$ be a non-backtracking random walk on the vertices. Let $O_j$ and $I_j$ respectively denote the out- and in-edges of a vertex $j$.

    When $k = 0$, we have  
    \[
    D^{-1}TS = D^{-1}A = P = \Tilde{P}^{(1)},
    \]  
    where $A$ is the adjacency matrix.  

    Now, for $k \geq 1$, observe that  
    \begin{align*}
        (D^{-1}TP_{nb}^kS)_{ij} &= \sum_{(x, y)} (D^{-1}T)_{i,(x,y)}(P_{nb}^kS)_{(x,y), j} \\
        &= \sum_{y} \frac{1}{d_i} (P_{nb}^kS)_{(i, y), j} \\
        &= \sum_{y}\frac{1}{d_i}\sum_{a} (P_{nb}^k)_{(i,y),(a, j)} \\
        &= \sum_{a}\sum_{y} \frac{1}{d_i} (P_{nb}^k)_{(i,y),(a,j)} \\
        &= \sum_{a}\P(W_{k+1} = (a, j) \mid W_0 \in O_i) \\
        &= \P(W_{k+1} \in I_j \mid W_0 \in O_i) \\
        &= \P(X_{k+1} = j \mid X_0 = i) \\
        &= \Tilde{P}_{ij}^{(k+1)}.
    \end{align*}
    The proof for $T^T$ follows similarly.  
\end{proof}

The above expression facilitates a matrix-theoretic approach to developing an analogue of the fundamental matrix. Our proposed definition is inspired by the relation 

\begin{align}\label{exp: Z power series}
Z = \sum_{k=0}^\infty (P - \mathbb{1}\pi^T)^k = I+\sum_{k=1}^\infty (P^k - \mathbb{1}\pi^T),
\end{align}
as established in \cite{grinstead1997introduction}.

\begin{observation}\label{obs: series expansion for defining Z}
    Consider a similar series to (\ref{exp: Z power series}) for the non-backtracking random walk, using the $k$-step transition probability matrix $\Tilde{P}^{(k)}$. \begin{align*}
    I+\sum_{k=1}^\infty (\Tilde{P}^{(k)} - \mathbb{1}\pi^T) &= I + P - \mathbb{1}\pi^T + D^{-1}T\sum_{k=1}^\infty (P_{nb}^k - \mathbb{1}\pi_e^T)S && (\text{Proposition~\ref{prop: vertex and edge transition prob relation}}) \\
    &= I + P - \1\pi^T + D^{-1}T\sum_{k=1}^\infty (P_{nb} - \1\pi_e^T)^kS \\
    &= I + P - \mathbb{1}\pi^T + D^{-1}T[(I - P_{nb} + \mathbb{1}\pi_e^T)^{-1} - I]S \\
    &= I + D^{-1}T(I - P_{nb} + \mathbb{1}\pi_e^T)^{-1}S - \1\pi^T.
\end{align*}

Similar work shows that 
\begin{align*}
    I+\sum_{k=1}^\infty (\Tilde{P}^{(k)} - \mathbb{1}\pi^T) &= D^{-1}T(I - P_{nb} + \mathbb{1}\pi_e^T)^{-1}T^T.
\end{align*}  

Both results rely on the fact that  
\[
\sum_{k=0}^\infty (P_{nb} - \mathbb{1}\pi_e^T)^k = (I - (P_{nb} - \mathbb{1}\pi_e^T))^{-1},
\]  
which holds as the spectral radius $\rho(P_{nb} - \mathbb{1}\pi_e^T) < 1$. This matrix inverse is the fundamental matrix for the non-backtracking random walk on the edge space of $G$.
\end{observation}

Taking this convergent power series as the definition of the vertex-space NBRW fundamental matrix, we define the following:

\begin{definition}\label{def: NB vertex fund matrix}
    Let $W_v = \mathbb{1}\pi^T$, $W_e = \mathbb{1}\pi_e^T$, and note that $(I - P_{nb} + W_e)^{-1} = Z^{(nb)}_e$. Define the vertex-space fundamental matrix for the non-backtracking random walk on $G$ as  
    \[
    Z^{(nb)}_v = I + D^{-1}TZ^{(nb)}_eS - W_v = D^{-1}TZ^{(nb)}_eT^T.
    \]  
    Additionally, define the \textit{fundamental matrix Kemeny's constant} for a non-backtracking random walk on the vertices of a graph as  
    \[
    \widehat\kemeny^{(nb)}_v(G) = \text{tr}(Z^{(nb)}_v) - 1.
    \]
\end{definition}

Thus, 
$Z^{(nb)}_v$ can be interpreted as a projection of the edge-space fundamental matrix, $Z^{(nb)}_e$, into the vertex space. The case for why this makes sense as a definition for Kemeny's constant in this setting is strengthened by the following fact, the proof of which follows  similarly those of Proposition \ref{prop: vertex and edge transition prob relation} and Observation \ref{obs: series expansion for defining Z}.

\begin{proposition}\label{prop: Z = projection for SRW}
Let $G$ be a graph, and consider the simple random walk on the vertices of $G$ and the edges of $G$. Letting $Z$ and $Z_e$ denote the fundamental matrices calculated using the transition matrices $P$ and $P_e$, it holds that \[
    Z = I + D^{-1}TZ^eS - W_v = D^{-1}TZ^eT^T.
    \]
\end{proposition}

The proposed definition for $Z^{(nb)}_v$ naturally prompts the question of which properties of the standard fundamental matrix are preserved in this context.

\begin{proposition}\label{prop: properties of Znb}
    For the matrix defined in Definition \ref{def: NB vertex fund matrix}, we have $Z^{(nb)}_v\1 = \1$, $\pi^TZ^{(nb)}_v=\pi^T$, and thus $W_vZ^{(nb)}_v = W_v$.
\end{proposition}
\begin{proof}
    We first show that $Z^{(nb)}_v\1 = \1$. Observe
    \begin{align*}
        Z^{(nb)}_v\1_n &= (I + D^{-1}TZ^{(nb)}_eS - W_v)\1_n \\
        &= D^{-1}TZ^{(nb)}_eS\1_n \\
        &= D^{-1}TZ^{(nb)}_e\1_{2m} \\
        &= D^{-1}T\1_{2m} \\
        &= \1_n.
    \end{align*}

    Next, we show that 
    $\pi^TZ^{(nb)}_v = \pi^T$. Observe
    \begin{align*}
        \pi^TZ^{(nb)}_v &= \pi^TD^{-1}TZ^{(nb)}_eT^T \\
        &= \frac{1}{2m}\1_n^TTZ^{(nb)}_eT^T \\
        &= \frac{1}{2m}\1_{2m}^T(I + \sum_{k=1}^\infty (P_{nb} - W_e)^k)T^T \\
        &= \frac{1}{2m}\1_{2m}^T(I + \sum_{k=1}^\infty(P_{nb}^k - W_e))T^T \\
        &= \frac{1}{2m}\1_{2m}^TT^T \\
        &= \pi^T.
    \end{align*}
\end{proof}

Finally, we compare the computational costs of our two definitions and show that the trace definition is more efficient. Recall from \eqref{eq:Mnb_v_related_to_Mev} that $M^{(nb)}_v = D^{-1}TM^{(nb)}_{ev}$, where the main bottleneck is forming $M^{(nb)}_{ev}$ by solving $n$ linear systems in the edge space. If we assume a $O(k^3)$ cost to invert a $(k\times k)$ matrix, then constructing $M^{(nb)}_v$ is $O(nm^3)$. Left- and right-multiplication by the vector $\pi$ each only cost $O(n)$, so evaluating $\kemeny^{(nb)}_v = \pi^\top M_v^{(nb)}\pi$ remains $O(nm^3)$.

By contrast, the trace definition $\widehat{\kemeny}^{(nb)}_v = \tr{Z^{(nb)}_v} - 1$ requires inverting a \( 2m \times 2m \) matrix and a few matrix multiplications, all of which are bounded by $O(m^3)$. Thus, the trace-based approach reduces the dependence on $n$, yielding an overall cost of $O(m^3)$ and offering significant speedups when $n$ is large.


\subsection{Sufficient Conditions for Consistency}
\label{subsec: 32}

Given the two proposed definitions of Kemeny’s constant for a non-backtracking random walk in the vertex space of a graph, it is natural to ask whether they ever coincide, and for what graphs they do coincide. In this section we prove that this is the case under the assumption of Condition \ref{cond: locally uniform return time}.



We present several lemmas that will be key to proving the main result. The following lemma establishes a commuting relationship between the matrices $W_v, W_e,$ and $D^{-1}T$. This identity will later enable us to interchange summation orders in key matrix products, simplifying the proof of Theorem \ref{thm:M Z relation equation}. 

\begin{lemma}\label{lem: cycling W and Wv}
    For the matrices $D$, $T$, $W_v = \1\pi^T$, and $W_e = \1\pi_e^T$ it holds that \[
    D^{-1}TW_e = W_vD^{-1}T.
    \]
\end{lemma}
\begin{proof}
    Fix $i, j, k \in V$. Then matrix multiplication gives that \begin{align*}
        (W_vD^{-1}T)_{i, (j, k)} &= \sum_{y\in V}\frac{d_i}{2m}(D^{-1}T)_{y, (j, k)} \\
        &= \frac{d_i}{2m}\frac{1}{d_i}\\
        &= \frac{1}{2m}.
    \end{align*}
    Similarly, \begin{align*}
        (D^{-1}TW_e)_{i, (j, k)} &= \sum_{(x, y) \in \hat{E}}(D^{-1}T)_{i, (x, y)}(W_e)_{(x, y), (j, k)} \\
        &= \sum_{y \sim i}\frac{1}{d_i}\frac{1}{2m} \\
        &= \frac{1}{2m}.
    \end{align*}
\end{proof}

The next lemma connects the diagonal matrix $R$ of non-backtracking return times to the edge-to-vertex hitting matrix $M^{(nb)}_{ev}$ and is a direct result of work presented in \cite{fasino2023hitting}.

\begin{lemma}\label{lem: R expression}
    Let $R$ denote the $n\times n$ diagonal matrix of non-backtracking mean return times. We can express $R$ as \[
    R = \mathrm{diag}(J + D^{-1}TP_{nb}M^{(nb)}_{ev}).
    \]
\end{lemma}
\begin{proof}
   From \cite[Section 4.2]{fasino2023hitting}, we have
    \[
    r_{ii} = 1 + \sum_{j, k \in V} p'_{ij}p_{i, j, k}m_{(j, k), i}
    \]
    where $p_{i,j, k} = P_{{nb}_{(i, j), (j, k)}}$ and $p'_{ij} = \frac{1}{\deg(i)}$ if $i \sim j$, and 0 otherwise. This gives the claimed matrix equation.  
\end{proof}

Next, we leverage the assumption of Condition \ref{cond: locally uniform return time} to derive a relationship satisfied by the matrix $M^{(nb)}_{ev}$ in terms of well-understood matrices.

\begin{lemma}\label{lem: Mev expression with T^TR}
    Let $G$ satisfy Condition \ref{cond: locally uniform return time}. Then it holds that \[
    M^{(nb)}_{ev} = J_{2m, n} + P_{nb}M^{(nb)}_{ev} - T^TR.
    \]
\end{lemma}
\begin{proof}
    Theorem 4.3 of \cite{fasino2023hitting} gives \[
    (M^{(nb)}_{ev})_{(i, j), k} = \begin{cases}
        0, & \text{if } i = k; \\
        1, & \text{if } j = k \text{ and } i \neq k; \\
        1 + \sum_{\ell\in V}p_{i, j, \ell}(M^{(nb)}_{ev})_{(j, \ell), k}, & \text{if } i, j \neq k.
    \end{cases}
    \]
    In matrix form, this equation gives 
    \begin{equation} \label{eq:Mev_with_Y}
            M^{(nb)}_{ev} = J_{2m, n} + P_{nb}M^{(nb)}_{ev} - Y
    \end{equation}
    for some $2m\times n$ matrix $Y$, where the entries are given by \[
    Y_{(i, j), k} = \begin{cases}
        0, & \text{if } k \neq i; \\
        1 + \sum_{\ell\in V}p_{i, j, \ell}(M^{(nb)}_{ev})_{(j, \ell), i}, & \text{if } k = i.
    \end{cases}
    \]
    Using Lemma \ref{lem: R expression}, notice that \[
    (T^TR)_{(i, j), k} = \begin{cases}
        0, & \text{if } k \neq i; \\
        1 + \sum_{\ell\in V} \sum_{y\in V} p'_{iy}p_{i, y, \ell}(M^{(nb)}_{ev})_{(y, \ell), i}, & \text{if } k=i.
    \end{cases}
    \]
    Notice that $\sum_{\ell\in V}p_{i, y, \ell}(M^{(nb)}_{ev})_{(y, \ell), i}$ represents the average time it takes to return to vertex $i$ given that the chain began with $X_0 = i, X_1 = y$. As such, we can write 
    \[\sum_{\ell\in V}p_{i, y, \ell}(M^{(nb)}_{ev})_{(y, \ell), k} = \E[\Tilde{T}^+_i \mid X_1 = y, X_0 = i] - 1\], where the $-1$ accounts for the fact that the walk in the edge space from  $O_i$ to $I_i$ is one less than its corresponding walk in the vertex space from $i \to i$. By our assumption, \begin{align*}
        (T^TR)_{(i, j), i} &= 1 + \sum_{y\in V}p_{iy}'\sum_{\ell\in V}p_{i, y, \ell}(M^{(nb)}_{ev})_{(y, \ell), i} \\
        &= 1 + \sum_{y\in V}p_{iy}'(\E[\Tilde{T}_i^+ \mid X_1 = j, X_0=i]-1) \\
        &= 1+ \E[\Tilde{T}_i^+ \mid X_1 = j, X_0=i]-1 \\
        &= 1 + \sum_{\ell\in V} p_{i, j, \ell}(M^{(nb)}_{ev})_{(j, \ell), i}.
    \end{align*}
    This shows that $T^TR = Y$ and completes the proof.
\end{proof}


The next result resolves a technical challenge: commuting operators between vertex- and edge-space fundamental matrices ($Z^{(nb)}$ and $Z^{(nb)}_e$). This commutation is made possible through Lemma \ref{lem: Mev expression with T^TR} and is essential for bridging vertex- and edge-centric formulations of Kemeny's constant.

\begin{lemma}\label{lem: Znb and Znb_e 'commuting'}
Let $G$ satisfy Condition \ref{cond: locally uniform return time}. Then \[
    Z^{(nb)}D^{-1}T(I - P_{nb})M^{(nb)}_{ev} = D^{-1}TZ^{(nb)}_e(I - P_{nb})M^{(nb)}_{ev}
\]    
\end{lemma}
\begin{proof}
    To reduce clutter, define $X = (I - P_{nb})M^{(nb)}_{ev}$. Expand the left hand side as \begin{align*}
        Z^{(nb)}D^{-1}TX &= [I - W_v + D^{-1}TZ^{(nb)}_eS]D^{-1}TX \\
        &= [(I-W_v)D^{-1}T + D^{-1}TZ^{(nb)}_eSD^{-1}T]X \\
        &= D^{-1}T[I - W_e + Z^{(nb)}_e((I-D_e^{-1})P_{nb} + D_e^{-1}\tau)]X.
    \end{align*}
    Note that the last equality comes from observing that \[
    SD^{-1}T = D_{e}^{-1}ST = D_e^{-1}[(D_e - I)P_{nb} + \tau] = (I - D_e^{-1})P_{nb} + D_e^{-1}\tau,
    \]
    which is justified through Lemma 2.8 of \cite{breen2023kemeny} and manipulating definitions. Then, the desired statement is true if and only if \[
    D^{-1}T[I - W_e + Z^{(nb)}_e((I-D_e^{-1})P_{nb} + D_e^{-1}\tau)]X = D^{-1}TZ^{(nb)}_eX.
    \]
    We will show that \[
    [I - W_e + Z^{(nb)}_e((I - D_e^{-1}) + D_e^{-1}\tau)]X = Z^{(nb)}_eX
    \]
    which will give the result. Multiplying by $(Z^{(nb)}_e)^{-1}$ we see that the above equation is equivalent to \begin{align*}
        [(Z^{(nb)}_e)^{-1}(I - W_e) + (I - D_e^{-1})P_{nb} + D_e^{-1}\tau]X &= X \\
        [(I - P_{nb}) + (I - D_e^{-1})P_{nb} + D_e^{-1}\tau]X &= X \\
        [I - D_e^{-1}P_{nb} + D_e^{-1}\tau]X &= X \\
        (\tau - P_{nb})X &= 0 \\
        (\tau - P_{nb})(I - P_{nb})M^{(nb)}_{ev} &= 0.
    \end{align*}
    Now using Lemma \ref{lem: Mev expression with T^TR} in the above statement gives \begin{align*}
        (\tau - P_{nb})(I - P_{nb})M^{(nb)}_{ev} &= (\tau - P_{nb})(J - T^TR) \\
        &= -\tau T^TR + P_{nb}T^TR.
    \end{align*}
    We show that this is equal to $0$. Straightforward matrix multiplication gives \begin{align*}
        (\tau T^TR)_{(i, j), v} &= \begin{cases}
            R_{jj}, & \text{if } j = v; \\
            0, & \text{if } j \neq v;
        \end{cases} \\
        (P_{nb}T^TR)_{(i, j), v} &= \begin{cases}
            R_{jj}, & \text{if } j = v; \\
            0, & \text{if } j \neq v.
        \end{cases}
    \end{align*}
    Thus $(\tau - P_{nb})(I - P_{nb})M^{(nb)}_{ev} = 0$ and we have the desired result.
\end{proof}

We are now ready for the main result, which is inspired by Theorem 11.16 from \cite{grinstead1997introduction}.

\begin{theorem}\label{thm:M Z relation equation}
    Let $G$ be such that Condition \ref{cond: locally uniform return time} holds. Then
    \[
    M_{v}^{nb} = J - Z^{(nb)}R + W_vM^{(nb)}_v.
    \]
    In particular, 
    \[
    m_{ij} = \frac{z_{jj} - z_{ij}}{\pi_j}.
    \]
\end{theorem}
\begin{proof}
    By Lemma \ref{lem: Mev expression with T^TR} we have
    \begin{align*}
        M^{(nb)}_{ev} &= J_{2m, n} + P_{nb}M^{(nb)}_{ev} - T^TR\\
        D^{-1}TM^{(nb)}_{ev} &= D^{-1}T[J_{2m, n} + P_{nb}M^{(nb)}_{ev} - T^TR]\\
        D^{-1}TM^{(nb)}_{ev} &= J_{n,n} + D^{-1}TP_{nb}M^{(nb)}_{ev} - R\\
        D^{-1}T(I - P_{nb})M^{(nb)}_{ev} &= J - R\\
        Z^{(nb)}D^{-1}T(I - P_{nb})M^{(nb)}_{ev} &= J - Z^{(nb)}R\\
        D^{-1}TZ_{e}^{nb}(I - P_{nb})M^{(nb)}_{ev} &= J - Z^{(nb)}R &\text{(Lemma \ref{lem: Znb and Znb_e 'commuting'})}\\
        D^{-1}T(I - W_e)M^{(nb)}_{ev} &= J - Z^{(nb)}R &\text{(Markov Chain fact)}\\
        (I - W_v)D^{-1}TM^{(nb)}_{ev} &= J - Z^{(nb)}R &\text{(Lemma \ref{lem: cycling W and Wv})}\\
        M^{(nb)}_v &= J - Z^{(nb)}R + W_vM^{(nb)}_v.
    \end{align*}
    From this equation we see that
    \begin{equation}\label{eq:mij relation}
    m_{ij} = 1 - z_{ij}r_j + (\pi^TM^{(nb)}_v)_{j}.
    \end{equation}
    Also, $m_{ii} = 0$ gives 
    \begin{equation*}
    0 = 1 - z_{ii}r_i + (\pi^TM^{(nb)}_v)_i,
    \end{equation*}
    hence
    \begin{equation}\label{eq:mii relation}
        (\pi^TM^{(nb)}_v)_i = z_{ii}r_i - 1.
    \end{equation}
    Combining equations \eqref{eq:mij relation} and \eqref{eq:mii relation} gives
    \[
    m_{ij} = (z_{jj} - z_{ij})r_j.
    \]
    For the non-backtracking random walk, it also holds that $r_j = 1/\pi_j$ (see \cite[Theorem 5.5]{fasino2023hitting}) which finishes the proof. 
\end{proof}

\begin{corollary}\label{cor:Z Kemeny edge-transitive}
    Let $G$ satisfy Condition \ref{cond: locally uniform return time}. Then
    \[
    \pi^TM^{(nb)}_v\pi = \tr{Z^{(nb)}} - 1.
    \]
\end{corollary}
\begin{proof}
    Consider the sum
    \begin{align*}
        \sum_{j\in V}m_{ij}\pi_j &=\sum_{j\in V}(z_{jj} - z_{ij})r_j\pi_j\\
        &= \sum_{j\in V}(z_{jj} - z_{ij}) &\text{(by Theorem \ref{thm:M Z relation equation})}\\
        &= \sum_{j\in V}z_{jj} - 1 &\text{(since $Z^{(nb)}\1 = \1$)}\\
        &= \tr{Z^{(nb)}} - 1.
    \end{align*}
    As this did not depend on the choice of $i\in V$, we get
    \[
    \sum_{i\in V}\sum_{j\in V}\pi_im_{ij}\pi_j = \left(\tr{Z^{(nb)}} - 1\right)\sum_{i\in V} \pi_i = \tr{Z^{(nb)}} - 1.
    \]
\end{proof}

Corollary \ref{cor:Z Kemeny edge-transitive} shows that the two proposed definitions for Kemeny's constant for the non-backtracking random walk in the vertex space collapse conveniently into a single expression under the assumption that $G$ has locally uniform return times. It is of interest to explore how far away these definitions may drift in general, and this is explored in the next section.

\subsection{Characterizing the gap for general graphs}\label{sec:gap}

In general, the non-backtracking random walk on the vertices of an arbitrary connected graph \( G \) does not exhibit locally uniform return times (see Condition~\ref{cond: locally uniform return time}). Nevertheless, it remains of interest to examine the proposed formulations of the NBRW Kemeny’s constant $\kemeny^{(nb)}_v(G)$ and $\widehat\kemeny^{(nb)}_v(G)$. However, it is not immediately clear how far these values may deviate from one another. In this section, we define an error matrix that directly connects the two formulations, \( \pi^\top M_v^{(nb)} \pi \) and \( \mathrm{tr}(Z_v^{(nb)}) - 1 \). While we do not derive explicit bounds on this gap, we view this construction as an important step toward understanding the numerical discrepancy between the two expressions.

The methods in this section build on the results of Lemmas~\ref{lem: Mev expression with T^TR} and~\ref{lem: Znb and Znb_e 'commuting'}, as well as Theorem~\ref{thm:M Z relation equation}, but now consider graphs \( G \) that may not satisfy Condition~\ref{cond: locally uniform return time}.

First, recall that in the proof of Lemma~\ref{lem: Mev expression with T^TR} we had the relation 
\begin{equation} \label{eq:Mev_expression_with_Y}
    M^{(nb)}_{ev} = J_{2m,n} + P_{nb}M^{(nb)}_{ev} - Y
\end{equation}
for the matrix $Y$ given by 
\begin{equation} \label{eq:Y_matrix}
    Y_{(i, j), k} = \begin{cases}
    0, & \text{if } k\neq i; \\
    1 + \sum_{\ell\in V}p_{i, j, \ell}(M^{(nb)}_{ev})_{({j}, \ell), i}, & \text{if } k = i.
\end{cases}
\end{equation}
When $G$ had locally uniform return times, then we saw that $Y$ was equal to $T^TR$, but this is not satisfied for a general $G$. With this in mind, consider the next lemma, which is an analogue of Lemma~\ref{lem: Znb and Znb_e 'commuting'}.

\begin{lemma} \label{lem: Znb and Znb_e 'commuting' with error}
    For a graph $G$, 
    \[
    Z^{(nb)}D^{-1}T(I-P_{nb})M^{(nb)}_{ev} = D^{-1}TZ^{(nb)}_e(I-P_{nb})M^{(nb)}_{ev} + D^{-1}TZ^{(nb)}_eD_{e}^{-1}F
    \]
    where
    \begin{equation} \label{eq:F_matrix}
        F := -(\tau - P_{nb})Y.
    \end{equation}
\end{lemma}
\begin{proof}
    In a similar manner to the proof of Lemma~\ref{lem: Znb and Znb_e 'commuting'}, define $X = (I - P_{nb})M^{(nb)}_{ev}$ for readability. Then observe:
    \begin{align*}
        (\tau - P_{nb})X &= F \\
        D_e^{-1}\tau X - D_e^{-1}P_{nb}X &= D_e^{-1}F \\
        D_e^{-1}\tau X - D_e^{-1}P_{nb}X + P_{nb}X &= D_e^{-1}F + P_{nb}X  \\
        D_e^{-1}\tau X + (I - D_e^{-1})P_{nb}X &= D_e^{-1}F + P_{nb}X \\
        D_e^{-1}\tau X + (I - D_e^{-1})P_{nb}X + (I-P_{nb})X &= D_e^{-1}F + P_{nb}X + (I - P_{nb})X \\
        D_e^{-1}\tau X + (I - D_e^{-1})P_{nb}X + (I-P_{nb})X &= D_e^{-1}F + X \\
        Z^{(nb)}_e[D_e^{-1}\tau X + (I - D_e^{-1})P_{nb}X + (I-P_{nb})X] &= Z^{(nb)}_e[D_e^{-1}F + X] \\
        Z^{(nb)}_e[D_e^{-1}\tau X + (I - D_e^{-1})P_{nb}X] + (I - W_{nb})X &= Z^{(nb)}_e[D_e^{-1}F + X] \\
        Z^{(nb)}D^{-1}TX &= D^{-1}TZ^{(nb)}_eX + D^{-1}TZ^{(nb)}_eD_e^{-1}F
    \end{align*}
    where the last line comes from reversing the chain of equalities from the first few lines of the proof of Lemma~\ref{lem: Znb and Znb_e 'commuting'}.
\end{proof}

\noindent
In comparison to Lemma~\ref{lem: Znb and Znb_e 'commuting'}, the additional term 
\[
E := D^{-1}TZ^{(nb)}_eD_{e}^{-1}F
\]
quantifies the deviation from equality. When \( G \) has locally uniform return times, we have \( F = 0 \), and thus \( E = 0 \), recovering the exact relationship from the earlier lemma.

The next result mirrors Theorem~\ref{thm:M Z relation equation}, but explicitly accounts for the newly defined error matrix $E$.

\begin{theorem} \label{thm:M_Z_relation_eqn_error}
    Let $G$ be a graph. Then 
    \[
    M^{(nb)}_v = J_{n,n} - Z^{(nb)}R + W_vM^{(nb)}_v - E
    \]
    where $E = D^{-1}TZ^{(nb)}_eD_e^{-1}F$ and $F$ is given in \eqref{eq:F_matrix}. In particular, 
    \[
    m_{ij} = \frac{z_{jj} - z_{ij}}{\pi_j} + E_{jj} - E_{ij}
    \]
\end{theorem}
\begin{proof}
    By Equations~\eqref{eq:Mev_expression_with_Y} and~\eqref{eq:Mnb_v_related_to_Mev}, 
    \begin{align*}
        D^{-1}TM^{(nb)}_{ev} &= D^{-1}T(J_{2m, n} + P_{nb}M^{(nb)}_{ev} - Y) \\
        M_{v}^{(nb)} &= J_{n, n} + D^{-1}TP_{nb}M^{(nb)}_{ev} - R.
    \end{align*}

    From these equations and again using the notational convenience $X = (I - P_{nb})M^{(nb)}_{ev}$, observe
    \begin{align*}
        Z^{(nb)}D^{-1}TX &= J - Z^{(nb)}R \\
        D^{-1}TZ^{(nb)}_eX + E &= J - Z^{(nb)}R && \text{(Lemma~\ref{lem: Znb and Znb_e 'commuting' with error})} \\
        D^{-1}T(I - W_e)M^{(nb)}_{ev} + E &= J - Z^{(nb)}R \\
        (I - W_v)M^{(nb)}_v + E &= J - Z^{(nb)}R && \text{(Lemma~\ref{lem: cycling W and Wv})} \\
        M^{(nb)}_v &= J - Z^{(nb)}R + W_vM^{(nb)}_v - E.
    \end{align*}
    This equation gives the relation $m_{ij} = 1 - z_{ij}r_j + (\pi^\top M^{(nb)}_v)_j - E_{ij}$. Combining this with the fact that $m_{ii} = 0$, work similar to that in the proof of Theorem~\ref{thm:M Z relation equation} gives the result.
\end{proof}

With this result established, we now give a relation between $\pi^\top M^{(nb)}_v \pi$ and $\text{tr}(Z_v^{(nb)}) - 1$, mirroring that of Corollary~\ref{cor:Z Kemeny edge-transitive}.

\begin{corollary} \label{cor:Z_kem_general}
    Let $G$ be a graph. Then
    \[
    \pi^\top M^{(nb)}_v \pi =\tr{Z_v^{(nb)}} - 1 + \frac{1}{2m}\tr{DE}
    \]
\end{corollary}
\begin{proof}
    By Theorem~\ref{thm:M_Z_relation_eqn_error}, we have that 
    \[
    \pi^\top M^{(nb)}_v \pi = \text{tr}(Z_v^{(nb)}) - 1 + \sum_{i \in V}\sum_{j \in V}\pi_i(E_{jj} - E_{ij})\pi_j
    \]
    Now observe, 
    \begin{align*}
        \sum_{i \in V}\sum_{j \in V}\pi_i(E_{jj} - E_{ij})\pi_j &= \frac{1}{4m^2}\sum_i\sum_j\deg(i)\deg(j)(E_{jj} - E_{ij}) \\
        & =\frac{1}{4m^2}\sum_j\deg(j)E_{jj}\sum_i\deg(i) - \frac{1}{4m^2}\sum_i\sum_j \deg(i)\deg(j)E_{ij} \\
        &= \frac{1}{2m}\sum_{j}\deg(j)E_{jj} - \frac{1}{4m^2}\1^\top DED\1 \\
        &= \frac{1}{2m}\text{tr}(DE) - \frac{1}{4m^2}\1^\top DED\1
    \end{align*}
    From here, we show that $\1^\top DED\1 = 0$. Begin by writing 
    \begin{align*}
        \1^\top DED &= \1^\top TZ^{(nb)}_eD_{e}^{-1}F \\
        &= \1^\top Z^{(nb)}_eD_e^{-1}F \\
        &= \1^\top D_e^{-1}F \\
        &= \1^\top D_e^{-1}(\tau-P_{nb})(-Y)
    \end{align*}
    Next, we show that $\1^\top D_e^{-1}(\tau - P_{nb}) = 0$. Writing out the matrix multiplication,
    \begin{align*}
        \left(\1^\top D_e^{-1}\tau\right)_{(i, j)} &= \sum_{(x,y)\in E}(\1^\top D_e^{-1})_{(x, y)}\tau_{(x, y), (i, j)} = \sum_{(x,y)}\frac{1}{\deg(y)}\tau_{(x, y), (i, j)} = \frac{1}{\deg(i)}
    \end{align*}
    and similarly, 
    \begin{align*}
        \left(\1^\top D_e^{-1}P_{nb}\right)_{(i, j)} &= \sum_{(x,y) \in E}(\1^\top D_e^{-1})_{(x,y)}(P_{nb})_{(x, y), (i, j)} \\
        &= \sum_{(x, i)}\frac{1}{\deg(i)}(P_{nb})_{(x,i), (i, j)} = \frac{1}{\deg(i)}\sum_{(x, i)}p_{x, i, j} = \frac{1}{\deg(i)}
    \end{align*}
    Thus, $\1^\top D_e^{-1}(\tau-P_{nb}) = 0$, implying $\1^\top D_e^{-1}(\tau-P_{nb})(-Y) = 0$, and so the result holds. 
\end{proof}

Although this provides a concise expression connecting the two proposed definitions, the error matrix $E$ remains analytically intractable in general, thereby complicating efforts to bound $\frac{1}{2m}\text{tr}(DE)$. Consequently, much of our subsequent analysis is conducted empirically. The following section explores this direction, along with broader investigations into the behavior of the NB Kemeny's constant.

\section{Examples and Formulae for $\mathcal{K}^{nb}(G)$ on Graph Families}\label{sec:ex}

While we do not precisely understand the gap between \( \left|\pi^TM^{(nb)}_v\pi - (\tr{Z^{(nb)}_v} - 1) \right| \) for the general connected graph $G$, numerical evidence (see Figure~\ref{fig:numerical_comparison}) suggests that the discrepancy is minor and that

\begin{equation} \label{eq:conjectured_inequality}
    \tr{Z^{(nb)}_v} - 1 \leq \pi^\top M^{(nb)}_v\pi.
\end{equation}

When determining explicit formulae of Kemeny's constant for several graph families, we found that deriving $\pi^\top M^{(nb)}_v\pi$ was much more tractable than the alternative, especially when there is potential to exploit cycles and symmetries in the graph. 
In this section we derive explicit formulae for $\pi^\top M^{(nb)}_v\pi$ for several families of graphs and compare the values to the SRW Kemeny's constant. Recall the notation that $\kemeny_{v}^{nb}(G) = \pi^\top M^{(nb)}_v \pi$. As suggested by Figure~\ref{fig:srw_vs_nbrw} and verified in each example, 

\begin{equation} \label{eq:nb_vs_srw_conjecture}
    \mathcal{K}_v^{nb}(G) \leq \mathcal{K}_v(G).
\end{equation}

\begin{figure}[H]
    \centering
    \includegraphics[width=\linewidth]{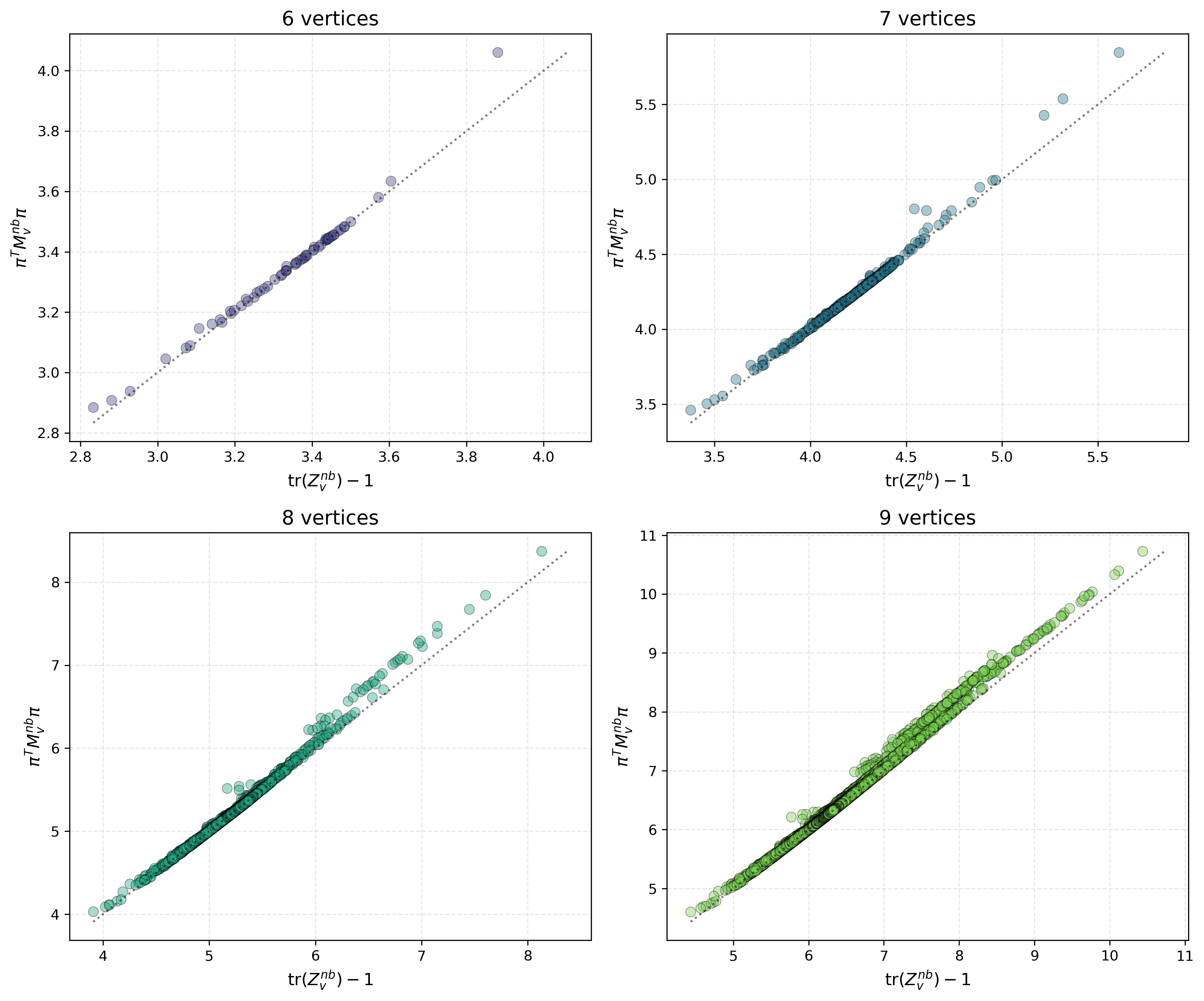}
    \caption{A numerical comparison on all connected graphs on 6-9 vertices between values of $\widehat{\kemeny}_{v}^{nb}(G) = \tr{Z^{(nb)}_v} - 1$ and $\kemeny_{v}^{nb}(G) = \pi^\top M^{(nb)}_v\pi$. Of note, these two values never seem to significantly stray from each other. Moreover, it suggests the relation $\tr{Z^{(nb)}_v} - 1 \leq \pi^\top M^{(nb)}_v\pi$. All computations done using SageMath.}
    \label{fig:numerical_comparison}
\end{figure}

\begin{figure}[H]
    \centering
    \includegraphics[width=\linewidth]{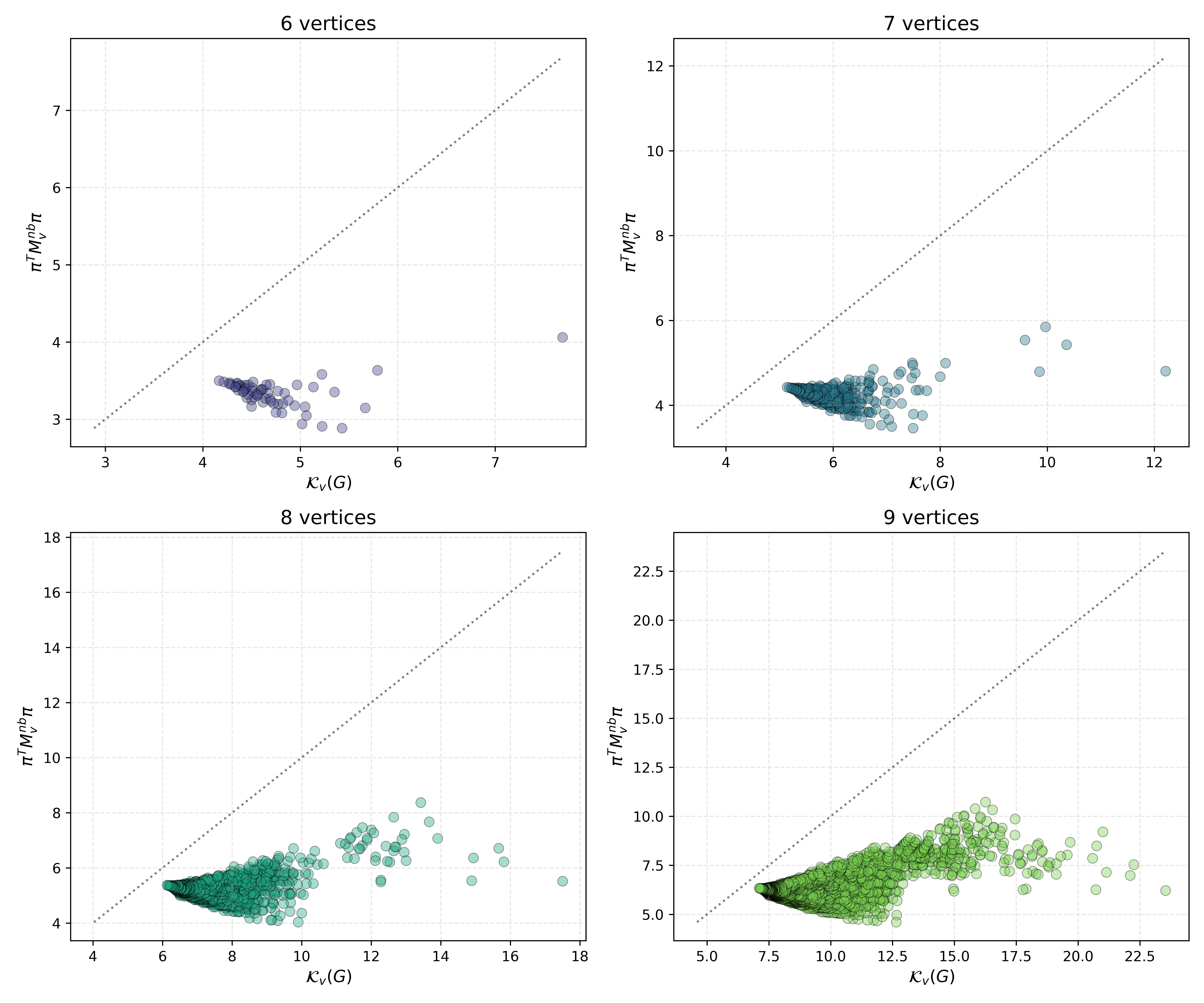}
    \caption{A numerical comparison of values of the SRW Kemeny's constant $\mathcal{K}_v(G)$ and the NB Kemeny's constant $\kemeny_{v}^{nb}(G) = \pi^\top M^{(nb)}_v\pi$ on all graphs with $|V(G)| \in \{6, 7, 8, 9\}$. Of interest, the NB value is strictly less than the SRW counterpart in each case.}
    \label{fig:srw_vs_nbrw}
\end{figure}

\subsection{$K_n$ vs. $C_nC_k$}
For the complete graph $K_n$, it is well known that $\mathcal{K}_v(K_n) \leq \mathcal{K}_v(G)$ for any other $G$ with $|V(G)| = n$ \cite{palacios2010bounds}. Interestingly, this does not hold in the non-backtracking case. Denote by $C_nC_k$ the cycle on $n$ vertices with an additional edge to create a cycle of size $k < n$. An example of this, $C_8C_3$ is shown in Figure \ref{fig:C8C3}. 

At a high level, the NBRW traverses cycles very quickly, becoming effectively deterministic after just one step. However, the cycle $C_n$ lacks a well-defined fundamental matrix $Z^{(nb)}_v$, as the non-backtracking edge-space transition matrix $P_{nb}$ is reducible and periodic. In particular, the NBRW on $C_n$ cycles deterministically through the edges, resulting in a transition matrix with multiple recurrent classes and no unique stationary distribution. If this pathological behavior were not present, we would expect $C_n$ to minimize NB Kemeny’s constant due to its fast mixing properties. Instead, the graph family $C_nC_3$ serves as a near-optimal structure: its small attached triangle introduces only a minimal perturbation, while the majority of the walk remains nearly deterministic, retaining much of the efficiency of the pure cycle.

\begin{definition} \label{def:CnCk}
    The graph $G = C_nC_k$ is the $n$-cycle, with the addition of an edge, forming an additional $k$-cycle. 
    Note that if $|V(G)|=n$ then $|E(G)|=n+1$.
\end{definition}

\begin{figure}[H]
    \centering
    \includegraphics[scale=0.5]{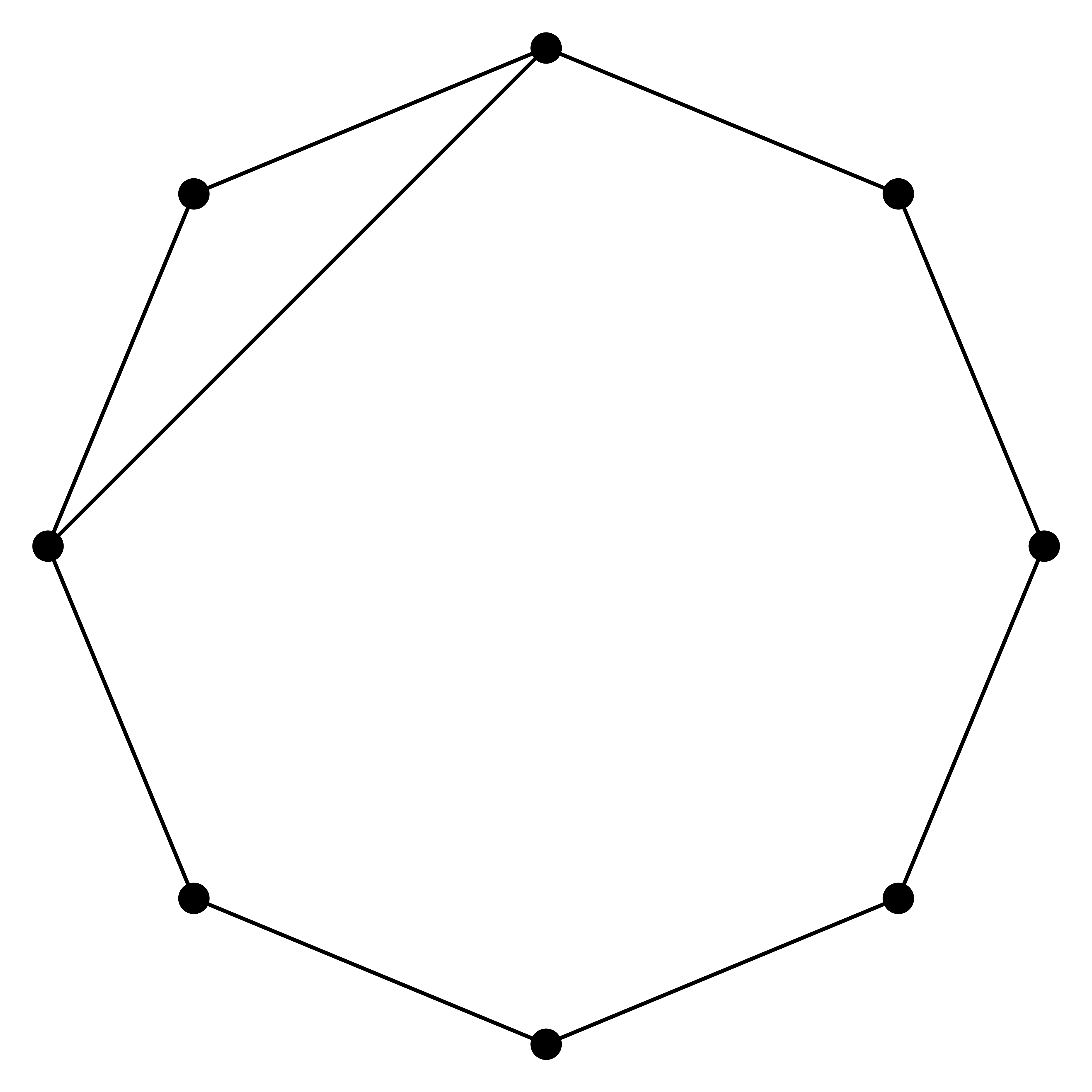}
    \caption{The graph denoted $C_8C_3$, an example of the family $C_nC_k$. This graph is formed by taking the $8$-cycle and adding an additional edge to create a cycle of length $3$. This graph is conjectured to minimize $\mathcal{K}^{(nb)}_v(G)$ for all $G$ with $|V(G)|=8$. In general, we conjecture that $C_nC_3$ minimizes $\mathcal{K}^{(nb)}_v(G)$ with $|V(G)| = n$.}
    \label{fig:C8C3}
\end{figure}

To support our claim that $C_nC_k$ has a smaller NB Kemeny's constant than $K_n$, and to compare this value to that of the SRW, we derive several intermediate results.

\begin{lemma} \label{lem:complete_kemeny}
    For the complete graph on $n$ vertices $K_n$, 
    \[
    \pi^TM^{(nb)}_v\pi = n - 3 + \frac{3}{n}.
    \]
\end{lemma}
\begin{proof}
    First recall that even in the NBRW case, $\pi_i = \frac{\deg(i)}{\vol(G)} = \frac{n-1}{n(n-1)} = \frac{1}{n}$. Given the symmetry of the complete graph, $(M^{(nb)}_v)_{ij} = \alpha$ for all $i \neq j$ and $0$ when $i=j$. Thus $\pi^\top M^{(nb)}_v\pi = \frac{(n-1)\alpha}{n}$. To find the value of $\alpha$, we consider the probability of getting from a vertex $i$ to a different vertex $j$ after $k$ non-backtracking steps. \\

    \begin{table}[htbp]
    \centering
    \resizebox{\textwidth}{!}{%
    \begin{tabular}{|c|c|c|c|c|c|c|c|}
    \hline
    \textbf{Step} & 1 & 2 & 3 & 4 & $\dots$ & $k$ & $\dots$ \\
    \hline
    $\mathbb{P}$ 
    & $\dfrac{1}{n-1}$ 
    & $\dfrac{n-2}{n-1} \cdot \dfrac{1}{n-2}$ 
    & $\dfrac{n-2}{n-1} \cdot \dfrac{n-3}{n-2} \cdot \dfrac{1}{n-2}$ 
    & $\dfrac{n-2}{n-1} \cdot \left( \dfrac{n-3}{n-2} \right)^2 \cdot \dfrac{1}{n-2}$ 
    & $\dots$ 
    & $\dfrac{n-2}{n-1} \cdot \left( \dfrac{n-3}{n-2} \right)^{k-2} \cdot \dfrac{1}{n-2}$ 
    & $\dots$ \\
    \hline
    \end{tabular}%
    }
    \caption{Probability of first hitting a target vertex at step \( k \) in a non-backtracking random walk on \( K_n \)}
    \label{tab:Kn_hitting_times}
    \end{table}

    Thus in general we have \begin{align*}
        \alpha = (M^{(nb)}_v)_{ij} = \frac{1}{n-1} + \left(\frac{n-2}{n-1}\right)\left(\frac{1}{n-2}\right)\sum_{k=2}^\infty k \left(\frac{n-3}{n-2}\right)^{k-2} = n + \frac{1}{n-1} - 2.
    \end{align*}
    Simplifying gives the result.
\end{proof}

Recalling that \( \mathcal{K}_v(K_n) = \frac{n^2 - 2n + 1}{n} = n - 2 + \frac{1}{n} \), we find preliminary support for the conjecture \( \mathcal{K}_v^{nb}(G) \leq \mathcal{K}_v(G) \). In particular, the difference  
\[
\mathcal{K}_v(K_n) - \mathcal{K}_v^{nb}(K_n) = 1 - \frac{2}{n}
\]
is positive for all \( n > 2 \), suggesting that the non-backtracking variant consistently yields a smaller value.

\begin{lemma} \label{lem:CnCk_kemeny}
    For the graph $C_nC_k$, 
    \[
    \pi^\top M^{(nb)}_v\pi = \frac{10n^3 + n^2(3k+6)-n(3k^2-9k+22) - 3k^2 + 6k - 15}{18(n+1)^2}.
    \]
    Furthermore, this quantity is increasing in $k$ for $3\leq k\leq \frac{n}{2}$.
\end{lemma}
\begin{proof}
    The expression of Kemeny's constant is computed in a similar fashion to the proof of Lemma~\ref{lem:complete_kemeny}, but proceeds in six cases to fully compute each entry of $M^{(nb)}_v$. The work for each case is given in Appendix~\ref{apx:proof_cnck_kem}.

    To show the second part, define 
    \[
    f(k) := \frac{10n^3 + n^2(3k+6)-n(3k^2-9k+22) - 3k^2 + 6k - 15}{18(n+1)^2} = \frac{N(k)}{18(n+1)^2}.
    \]
    To show that $f(k) \leq f(k+1)$, it suffices to show that $N(k) \leq N(k+1)$. It can be shown that $N(k+1) - N(k) = -6kn + 6k + 3n^2 + 6n + 3$. The terms that depend on $k$ can be written as $6k(1 - n)$, which is a decreasing function of $k$ when $n \geq 2$. Given that $k \leq \frac{n}{2}$, we have \begin{align*}
        -6kn + 6k + 3n^2 + 6n + 3 & \geq 3(1-n)n + 3n^2 + 6n + 3 \\
        &= 9n + 3 \\
        &\geq 0
    \end{align*}
    for any nonnegative $n$. 
    
\end{proof}

This not only provides a closed-form expression for the NB Kemeny's constant of \( C_nC_k \), but also demonstrates that, among all choices of \( k \) for a fixed \( n \), the minimum is attained when \( k = 3 \).  We note in the following that this is in contrast to the situation with the simple Kemeny's constant.

\begin{lemma} \label{lem:CnCk_SRW_kemeny}
    For the graph $C_nC_k$, the SRW value of Kemeny's constant is given by 

    \[
    \textstyle
    \mathcal{K}_v(C_nC_k) = \frac{kn^4 - (2k^2 - 6k + 2)n^3 + (2k^3 - 12k^2 + 15k - 6)n^2 - (k^4 - 12k^3 + 27k^2 - 21k + 7)n - 4k^4 + 16k^3 - 23k^2 + 14k - 3}{6(n+1)(k(n+2) - k^2-1)}.
    \]
    Furthermore, this expression is decreasing in $k$
    for $3 \leq k \leq \frac{n}{2}$.
\end{lemma}
\begin{proof}
    We omit the full proof here, but acknowledge that the expression for Kemeny's constant relies on results from \cite{barrett2020spanning} to compute resistance distances and the inequality is proved from a similar, yet more technical argument as in the proof of the previous lemma. 
\end{proof}


As a corollary to Lemmas \ref{lem:complete_kemeny} and \ref{lem:CnCk_kemeny}, we can conclude the following with direct computation.
\begin{corollary} \label{thm:CnC3_less_Kn}
For $n\geq4$,
\[\mathcal{K}_v^{nb}(C_nC_3) < \mathcal{K}_v^{nb}(K_n).\]
\end{corollary}

This shows that, unlike in the SRW setting, the complete graph $K_n$ does not minimize the NB formulation of Kemeny's constant. Rather, for $n\geq 4$, the NB Kemeny's constant is strictly smaller for the graph $C_nC_3$. Moreover, Lemmas~\ref{lem:CnCk_kemeny} and~\ref{lem:CnCk_SRW_kemeny} highlight an interesting contrast between SRW and NBRW dynamics on $C_nC_k$: as $k$ increases, Kemeny's constant increases in the SRW case but decreases in the NBRW case. Furthermore, 
we observe that the NB formulation of Kemeny’s constant yields consistently smaller values than the SRW counterpart, both for the complete graph and for \( C_nC_k \).

\subsection{Cycle Barbells}
The authors of \cite{breen2023kemeny} examined both the SRW and NBRW formulations of edge-space Kemeny's constant for the cycle barbell family of graphs, observing that this family empirically maximizes the constant among all graphs with minimum degree 2 on \( n \) vertices and \( n+1 \) edges. In a similar vein, we observed that our proposed formulation of non-backtracking Kemeny's constant in the vertex space exhibits analogous behavior, up to order 9. Motivated by these findings, we derive a closed-form expression for \( \mathcal{K}_v^{\mathrm{nb}} \) for this family of graphs.

\begin{definition}
    The \textit{cycle barbell} $G = CB(k, a, b) = C_a \oplus P_k \oplus C_b$ is the $1$-sum of an $a$-cycle, a path on $k$ vertices, and a $b$-cycle for integers $k\geq 2, a,b \geq 3$. It is straightforward to verify that the number of vertices and edges are given by $|V(G)| = a + b + k - 2$ and $|E(G)| = a + b + k - 1$.
\end{definition}

\begin{figure}[htb]
    \centering
    \includegraphics[width=0.4\linewidth]{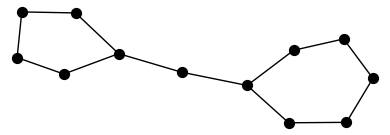}
    \caption{The graph CB(3, 5, 6).}
    \label{fig:cb_5_3_6}
\end{figure}

\begin{lemma} \label{lem:cycle_barbell}
    For the cycle barbell $CB(k, a, b)$, we have \begin{align*}
        \pi^\top M^{(nb)}_v\pi &= 
        \frac{13a^3 + 13b^3 + b^2(-87 + 60k) + a^2(-87 + 60b + 60k) + 2b(49 - 72k + 24k^2)}{24(-1 + a + b + k)^2} \\ &+ \frac{4(-9 + 20k - 15k^2 + 4k^3) + 2a(49 + 30b^2 - 72k + 24k^2 + 12b(-7 + 4k))}{24(-1 + a + b + k)^2}.
    \end{align*}
\end{lemma}
\begin{proof}
    This follows by computing the matrix $M^{(nb)}_v$. Full details are given in Appendix~\ref{apx:proof_cb_kem}.    
\end{proof}

\begin{lemma} \label{lem:cb_max}
    For a fixed number of vertices $n$, $\mathcal{K}^{(nb)}_v(CB(k,a,b))$ is maximized at $CB(2, \lceil n/2 \rceil, \floor{n/2})$.
\end{lemma}
\begin{proof}
    The proof of this involves restricting the focus to the numerator, $N(k, a, b)$ of the expression in Lemma \ref{lem:cycle_barbell} as for a fixed $n$, the denominator is fixed. Then we fix $k$ and reduce the numerator to a quadratic function of a single variable, $a$. We find that the maximizer $a^*$ must occur at either $\frac{n-k+2}{2}$ or at the end points of the feasible region for $a$, which is $3 \leq a \leq n-k-1$. In the case that $a^* = \frac{n-k+2}{2}$, we then show that $N(k, a^*, b^*)$ is a non-increasing function of $k$, which must be maximized when $k=2$. In the case that $a^*$ occurs at the end points, then we show that regardless of $k$, the value of $N(2, \lceil n/2 \rceil, \floor{n/2})$ is always larger and is thus the global maximum.
     The full details can be found in Appendix~\ref{apx:proof_cb_max}.
\end{proof}

This lemma mirrors the result of Theorem 5.7 in \cite{breen2023kemeny}, which pertains to the edge-based non-backtracking Kemeny's constant. The following lemma parallels Theorem 5.6 of \cite{breen2023kemeny}, which addresses the edge-based SRW Kemeny's constant. These observations suggest that the opposing behaviors arise as artifacts of the random walk type, rather than the choice of state space.

\begin{lemma}
    For a cycle barbell graph $CB(k, a, b)$ on a fixed number of vertices $n$, the simple random walk value of Kemeny's constant $\mathcal{K}_v$ is maximized at $CB(n-4, 3, 3)$.
\end{lemma}

\begin{proof}
    Theorems 5.2, 5.3, and 5.6 of \cite{breen2023kemeny} show that for any cycle barbell graph $G = CB(k, a, b)$,
    \[
    \mathcal{K}_v(G) = \mathcal{K}_e(G) - (n-2),
    \]
    where $\mathcal{K}_v$ is the vertex-based Kemeny's constant and $\mathcal{K}_e$ is the edge-based variant. They further prove that $\mathcal{K}_e(G)$ is maximized when $G = CB(n-4, 3, 3)$ among all cycle barbells on $n$ vertices.

    Since $n$ is fixed and the relation between $\mathcal{K}_v$ and $\mathcal{K}_e$ is linear with respect to $n$, it follows that $\mathcal{K}_v(G)$ is also maximized at $CB(n-4, 3, 3)$.
\end{proof}

Taken together, the preceding lemmas show that the SRW and NBRW variants of Kemeny’s constant behave quite differently on the cycle barbell family: the SRW value is maximized when the connector is long and the cycles are small, while the NBRW value is maximized when the connector is short and the cycles are balanced and large. 

It is also straightforward to verify symbolically (using a computer algebra system) that for all members of this family, the SRW value $\mathcal{K}_v(G)$ (given in \cite{breen2023kemeny}) is strictly greater than the NBRW value $\pi^\top M^{(nb)}_v \pi$, reinforcing the observation that the NBRW variant of Kemeny's constant consistently yields smaller values than its SRW counterpart.

\subsection{Complete Bipartite Graphs}
While the complete graph \( K_n \) minimizes the SRW Kemeny’s constant, the complete bipartite graph \( K_{p,q} \) also achieves a low value, \( \mathcal{K}_v(K_{p,q}) = p + q - \frac{3}{2} \). This combined with its incredibly symmetric nature makes it a natural case for examination.

We compute the NBRW Kemeny’s constant for \( K_{p,q} \). Since \( K_{p,q} \) is edge-transitive, the NBRW is locally uniform, and thus this value equals both \( \pi^\top M_v^{\mathrm{nb}} \pi \) and \( \mathrm{tr}(Z_{\mathrm{nb}}) - 1 \).

\begin{lemma} \label{lem:Kpq_kem}
    For the complete bipartite graph $K_{p, q}$, 
    \[
    \mathcal{K}_v^{(nb)}(K_{p,q}) = p + q + \frac{1}{p} + \frac{1}{q} - \frac{7}{2}.
    \]
\end{lemma}
\begin{proof}
    This is again proved through finding $\pi^TM^{(nb)}_v\pi$. We consider 4 cases. For clarity, denote the two disjoint vertex sets as $V_1, V_2$.
    \begin{enumerate}
        \item Suppose $i\in V_1, j\in V_2$. The probability of reaching j from i in one step is $\frac{1}{q}$. The probability of hitting in two steps, and similarly in any even number of steps, is 0. In 3 steps the probability is $\left(\frac{q-1}{q}\right)\left(\frac{1}{q-1}\right)$. It can similarly be shown that at the $(2k+1)$-th step, the probability of hitting j is $\left(\frac{q-1}{q}\right)\left(\frac{q-2}{q-1}\right)^{k-1}\left(\frac{1}{q-1}\right)$. Thus, we see that \begin{align*}
            \Tilde{m}_{ij} &= \frac{1}{q} + \frac{1}{q}\sum_{k=1}^\infty(2k+1)\left(\frac{q-2}{q-1}\right)^{k-1} \\
            &= \frac{(q-1)(2q-1)+1}{q}
        \end{align*}
        \item $i\in V_2, j\in V_1$. Reverse the roles in the previous argument to show that \[
        \Tilde{m}_{ij} = \frac{(p-1)(2p-1)+1}{p}
        \]
        \item $i, j \in V_1$. Consider again the probabilities of reaching j in a fixed number of steps (see Table~\ref{tab:Kpq_nb_hitting}).  \\

        \begin{table}[htbp]
        \centering
        \begin{tabular}{|c|c|c|c|c|c|c|}
        \hline
        \textbf{Step} & 1 & 2 & 3 & 4 & 5 & $\dots$ \\
        \hline
        $\mathbb{P}$ 
        & $0$ 
        & $\dfrac{1}{p-1}$ 
        & $0$ 
        & $\left( \dfrac{p-2}{p-1} \right) \cdot \dfrac{1}{p-1}$ 
        & $0$ 
        & $\dots$ \\
        \hline
        \end{tabular}
        \caption{non-backtracking hitting time probabilities in Case 3 for a walk on \( K_{p, q} \).}
        \label{tab:Kpq_nb_hitting}
        \end{table}

        It follows then that at the $(2k+1)$-th step, the probability of hitting j is $0$ and at the $2k$-th step, the probability is $\left(\frac{p-2}{p-1}\right)^{k-1}\left(\frac{1}{p-1}\right)$. Thus, \begin{align*}
            \Tilde{m}_{ij} &= \left(\frac{1}{p-1}\right)\sum_{k=1}^\infty (2k)\left(\frac{p-2}{p-1}\right)^{k-1} \\
            &= 2(p-1)
        \end{align*}
        \item $i, j \in V_2$. Without loss of generality, apply the same argument in Case 3 to get \[
        \Tilde{m}_{ij} = 2(q-1)
        \]
    \end{enumerate}
    Note also that \[
        \pi_i =
        \begin{cases}
        \frac{1}{2p} & \text{if } i \in \{1, \dots, p\} \\
        \frac{1}{2q} & \text{if } i \in \{p+1, \dots, p+q\}
        \end{cases}
\]
    It then follows that \begin{align*}
        \pi^TM^{(nb)}_v\pi &= \sum_{i=1}^p\sum_{j=p+1}^{p+q}\left(\frac{1}{2p}\right)^2\left(2(p-1)\right) + \sum_{i=1}^p\sum_{j=p+1}^{p+q}\left(\frac{1}{4pq}\right)\left(\frac{(q-1)(2q-1)+1}{q}\right) \\
        &+ \sum_{i=p+1}^{p+q}\sum_{j=1}^{p}\left(\frac{1}{4pq}\right)\left(\frac{(p-1)(2p-1)}{p}\right) + \sum_{i=p+1}^{p+q}\sum_{j=p+1}^{p+q}\left(\frac{1}{2q}\right)^2\left(2(q-1)\right) \\
        &= \frac{(p-1)^2}{2p} + \frac{(q-1)(2q-1)+1}{q} + \frac{(p-1)(2p-1)+1}{p} + \frac{(q-1)^2}{2q} \\
        &= p + \frac{1}{p} + q + \frac{1}{q} - \frac{7}{2}
    \end{align*}
\end{proof}

In this family, it is immediately clear that the NB formulation of Kemeny's constant is always less than its SRW formulation, as
\[
    \mathcal{K}_v(K_{p,q}) - \mathcal{K}_v^{nb}(K_{p, q}) = 2 + \frac{1}{p} + \frac{1}{q} > 0,
\]
which again supports the conjecture given in \eqref{eq:nb_vs_srw_conjecture}.

\subsection{$(n, k)$-Pinwheel Graphs}
As a final example, we consider a family of graphs formed as the $1$-sum of $k$ $n$-cycles. This case offers further empirical support for the conjecture in \eqref{eq:nb_vs_srw_conjecture}, suggesting that the NBRW Kemeny’s constant is consistently smaller than its SRW counterpart. Moreover, it serves to illustrate the proof techniques employed across several examples in this work.

\begin{definition} \label{def:nk_pinwheel}
    Define the $(n, k)$-pinwheel graph $PW(n, k) := C_n^{(1)} \oplus C_n^{(2)} \oplus \dots \oplus C_n^{(k)}$, that is, the graph obtained by identifying a single vertex in $k$ copies of the cycle $C_n$. Note that $|V(PW(n, k))| = kn - k + 1$ and $|E(PW(k, n))| = kn$. See Figure~\ref{fig:PW_4_3} for a depiction of $PW(4, 3)$.
\end{definition}

\begin{figure}[htb]
    \centering
    \includegraphics[scale=0.5]{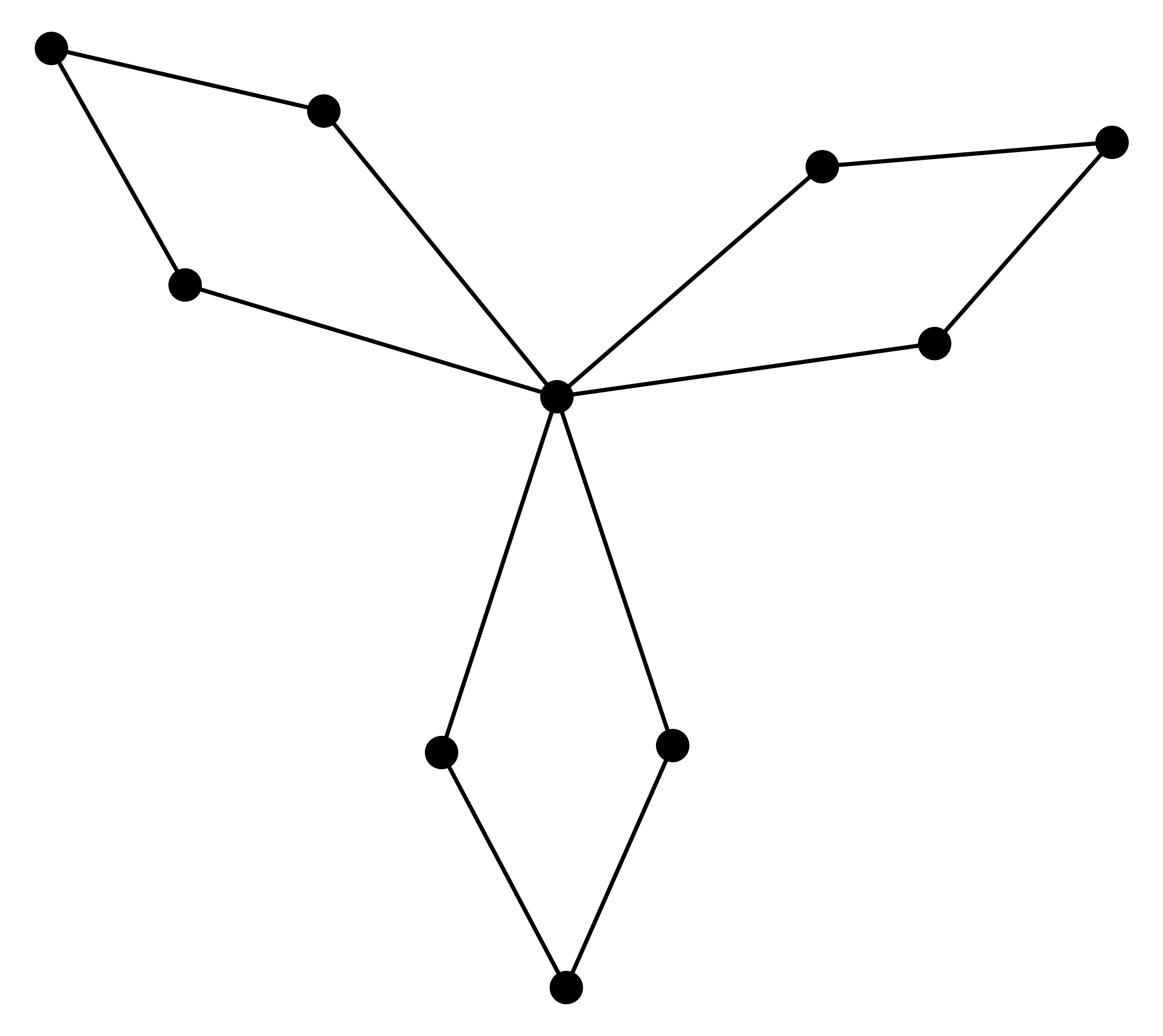}
    \caption{The Graph $PW(4, 3)$.}
    \label{fig:PW_4_3}
\end{figure}

\begin{lemma} \label{lem:pinwheel_nb_kem}
    For the graph $PW(n, k)$, 
    \[
    \pi^\top M^{(nb)}_v\pi = \frac{(n-1)(12k^3n - 18k^2n + 13kn - 2k - 4n + 2}{6kn(2k-1)}
    \]
\end{lemma}
\begin{proof}
    Begin by labeling the vertices of the pinwheel graph \( PW(n, k) \) as \( \{0, 1, \dots, nk - k\} \), where vertex \( 0 \) denotes the central hub, and each of the \( n \) peripheral \( k \)-cycles is assigned a contiguous block of \( k \) consecutive indices.

        To compute $\Tilde{m}_{pq}$, we consider the four cases: \begin{enumerate}
        \item[(i)] $p, q \in C_n^{(l)}$ with $d_p = d_q = 2$. \begin{enumerate}
            \item[(a)] $p < q$
            \item[(b)] $p > q$
        \end{enumerate}
        \item[(ii)] $p\in C_n^{(l)}, q\in C_n^{l'}$ with $l\neq l'$ and $d_p = d_q = 2$. 
        \item[(iii)] $p\in C_n^{(l)}, q = 0$ with $d_p = 2, d_q = 2k$. 
        \item[(iv)] $p = 0, q\in C_n^{(l)}$ with $d_p = 2k, d_q = 2$.
    \end{enumerate}

Now we compute hitting times associated with each one.

\begin{enumerate}
    \item[(i)]\begin{enumerate}
        \item[(a)] Considering all possible paths from p to q, we compute \begin{align*}
            \Tilde{m}_{pq} &= \frac{q-p}{2} + \left(\frac{p-(nl-l)}{2}\right)\left(\frac{n(l+1)-l-q}{2k-1}\right) \\ &+\sum_{m=0}^\infty\left(\frac{1}{2}\right)\left(\frac{2k-2}{2k-1}\right)\left(\frac{2k-3}{2k-1}\right)^m\left(\frac{1}{2k-1}\right)\bigg[p-(nl-l) + (m+1)n + n(l+1)-l-q\bigg] \\
            &+\sum_{m=0}^\infty\left(\frac{1}{2}\right)\left(\frac{2k-2}{2k-1}\right)\left(\frac{2k-3}{2k-1}\right)^m\left(\frac{1}{2k-1}\right)\bigg[p-(nl-l) + (m+1)n + q -(nl-l)\bigg] \\
            &= \frac{2k^2n-2kln+2kl-2kn+2kq+ln+n(l+1)-2l-2q}{4k-2}.
        \end{align*}
        \item[(b)] In a very similar computation to (a), it is shown that \begin{align*}
            \Tilde{m}_{pq} = \frac{2k^2n + 2l(k(n-1)+1) -2kq -2n(l+1)+n+2q}{4k-2}
        \end{align*}
    \end{enumerate}
    \item[(ii)] Again considering all possible paths from p to q, compute \begin{align*}
        \Tilde{m}_{pq} &= \sum_{m=0}^\infty\left(\frac{1}{2}\right)\left(\frac{2k-3}{2k-1}\right)^m\left(\frac{1}{2k-1}\right)\bigg[p-(nl-l)+mn+n(l'+1)-l'-q\bigg] \\
        &+ \sum_{m=0}^\infty\left(\frac{1}{2}\right)\left(\frac{2k-3}{2k-1}\right)^m\left(\frac{1}{2k-1}\right)\bigg[p-(nl-l)+mn+q-(nl'-l')\bigg] \\
        &+ \sum_{m=0}^\infty\left(\frac{1}{2}\right)\left(\frac{2k-3}{2k-1}\right)^m\left(\frac{1}{2k-1}\right)\bigg[n(l+1)-l-p+mn+n(l'+1)-l'-q\bigg] \\
        &+ \sum_{m=0}^\infty\left(\frac{1}{2}\right)\left(\frac{2k-3}{2k-1}\right)^m\left(\frac{1}{2k-1}\right)\bigg[n(l+1)-l-p+mn+q-(nl'-l')\bigg] \\
        &= n\left(k-\frac{1}{2}\right).
    \end{align*}
    \item[(iii)] In this case the NBRW will head straight towards q no matter the initial choice: \begin{align*}
        \Tilde{m}_{pq} &= \frac{1}{2}\left(p-(nl-l)\right) + \frac{1}{2}\left(n(l+1)-l-p\right) \\
        &= \frac{n}{2}.
    \end{align*}
    \item[(iv)] \begin{align*}
        \Tilde{m}_{pq} &= \frac{n(l+1)-l-q}{2k} + \frac{q-(nl-l)}{2k} \\
        &+ \sum_{m=0}^\infty\left(\frac{2k-2}{2k}\right)\left(\frac{2k-3}{2k-1}\right)^m\left(\frac{1}{2k-1}\right)\bigg[(m+1)n + n(l+1)-l-q\bigg]\\
        &+ \sum_{m=0}^\infty\left(\frac{2k-2}{2k}\right)\left(\frac{2k-3}{2k-1}\right)^m\left(\frac{1}{2k-1}\right)\bigg[(m+1)n + q-(nl-l)\bigg] \\
        &= kn -n + \frac{n}{2k}.
    \end{align*}
\end{enumerate}
Finally, note that $\pi = \begin{bmatrix}
    \frac{1}{n} & \frac{1}{kn} & \frac{1}{kn} & \dots & \frac{1}{kn}
\end{bmatrix}$ and calculate \begin{align*}
    \kemeny^{(nb)}_v(PW_{n,k}) &= \pi^TM^{(nb)}_v\pi \\
    &= \left(\frac{1}{kn}\right)^2\sum_{l=0}^{k-1}\sum_{p=nl-l+1}^{nl-l+n-2}\sum_{q=p+1}^{nl-l+n-1}\bigg[\frac{2k^2n-2kln+2kl-2kn+2kq+ln+n(l+1)-2l-2q}{4k-2}\bigg] \\
    &+ \left(\frac{1}{kn}\right)^2\sum_{l=0}^{k-1}\sum_{p=nl-l+2}^{nl-l+n-1}\sum_{q=nl-l+1}^{p-1}\bigg[\frac{2k^2n + 2l(k(n-1)+1) -2kq -2n(l+1)+n+2q}{4k-2}\bigg] \\
    &+ \left(\frac{1}{kn}\right)^2\sum_{p=1}^{kn-k}\sum_{q=1}^{kn-k-n+1}\bigg[n\left(k-\frac{1}{2}\right)\bigg] \\
    &+ \left(\frac{1}{kn^2}\right)\sum_{p=1}^{kn-k}\bigg[\frac{n}{2}\bigg] \\
    &+ \left(\frac{1}{kn^2}\right)\sum_{p=1}^{kn-k}\bigg[kn-n+\frac{n}{2k}\bigg] \\
    &= \frac{(n-1)(12k^3n-18k^2n+13kn-2k-4n+2)}{6kn(2k-1)}.
\end{align*}
\end{proof}

\begin{lemma} \label{lem:pinwheel_srw_kem}
    For the graph $G = PW(n, k)$, 
    \[
    \mathcal{K}_v(G) = \frac{(2k-1)(n^2-1)}{6}
    \]
\end{lemma}
\begin{proof}
    Corollary 2.4 of \cite{faught20221} gives that for $G = G_0 \oplus_v G_1 \oplus_v \dots \oplus_v G_{n-1}$ and $m_i = \lvert E(G_i)\rvert $, \begin{align*}
        \mathcal{K}(G) = \frac{\sum_{i=0}^{n-1}m_i\left(\mathcal{K}(G_i) + \sum_{j\neq i}\mu(G_i, v)\right)}{\sum_{i=0}^{n-1}m_i}
    \end{align*}
    Where $\mu(G_i, v) = \sum_{j\in V(G_i)}d_jr_{G_i}(j, v)$ is the $moment$ of $v\in V(G_i)$ and $r_{G_i}(j, v)$ is the effective resistance of $j$ to $v$ in $G_i$. Note that for the case at hand, $G_i = C_n$ and $m_i = n$ for all $i = 0, \dots, k-1$. It is well-known that  that $\mathcal{K}(C_n) = \frac{n^2-1}{6}$ \cite{bendito2008formula, wang2017kemeny}.  Finally, the last tool we will need is that $r_{C_n}(i, j) = \frac{d(i,j)(n-d(i,j))}{n}$ with $d(i,j)$ representing the minimum number of steps from node $i$ to node $j$, which is established in the proof of theorem 5.2 of \cite{breen2023kemeny}. Now with all of this, we can compute \begin{align*}
        \mathcal{K}(PW_{n,k}) &= \frac{\sum_{i=0}^{k-1}n\left(\mathcal{K}(C_n) + \sum_{j=0}^{k-2}\mu(C_n, v)\right)}{\sum_{i=0}^{k-1}n} \\
        &= \mathcal{K}(C_n) + (k-1)\mu(C_n, v) \\
        &= \frac{n^2-1}{6} + 2\sum_{j\in V(C_n)}\frac{d(j,v)(n-d(j,v))}{n} \\
        &= \frac{n^2-1}{6} + \frac{2}{n}\sum_{j=1}^{n-1}j(n-j) \\
        &= \frac{n^2-1}{6} + \frac{n^2-1}{3} \\
        &= \frac{(2k-1)(n^2-1)}{6}.
    \end{align*}
\end{proof}

From Lemmas~\ref{lem:pinwheel_nb_kem} and~\ref{lem:pinwheel_srw_kem}, we see that \( \mathcal{K}_v(PW(n, k)) \in O(kn^2) \), while \( \pi^\top M_v^{\mathrm{nb}} \pi \in O(kn) \). A straightforward verification using a computer algebra system confirms that the NBRW Kemeny’s constant is strictly smaller than its SRW counterpart.



\section{Conclusions and Conjectures}
We have proposed two definitions for a non-backtracking variant of Kemeny's constant of a graph.  We have shown that these match for edge-transitive graphs, and more generally for graphs satisfying the locally uniform return time property of Condition \ref{cond: locally uniform return time}.  This naturally leads to several questions that remain open.

\begin{question}
    Can we classify all graphs satisfying Condition \ref{cond: locally uniform return time}?
\end{question}

As observed earlier, all edge-transitive graphs satisfy this condition, but there are some examples of non-edge-transitive graphs that also satisfy it.  We do not know if these examples are sporadic, or fit somehow into larger infinite families of examples.  

While Condition \ref{cond: locally uniform return time} was sufficient for equality of our two proposed definitions, we do not know if there may be other graphs for which they are also equal.

\begin{question}
    Does the equality $\pi^\top M^{(nb)}_v\pi = \tr{Z^{(nb)}}-1$ imply that Condition \ref{cond: locally uniform return time} holds?
\end{question}

It is of further interest to understand how far apart the two proposed definitions can be.  

\begin{question}
    Can we bound $\pi^\top M^{(nb)}_v\pi - (\tr{Z^{(nb)}}-1)$?
\end{question}

Of course, Corollary \ref{cor:Z_kem_general} suggests that one answer to this question is given by the quantity $\frac{1}{2m}\tr{DE}$ (with $D$ and $E$ defined in Section \ref{sec:gap}).  However, these matrices are difficult to parse, and our results thus far do not give a clear picture of how large this trace can be. 

Of significant interest to us is the comparison between the values of Kemeny's constant for the non-backtracking random walk and the simple random walk.  Our examples have shown contrasting behavior between extremal examples for the NBRW and the SRW. Vertices of degree two seem to have a big impact on the NBRW Kemeny's constant.  We are particularly interested in the relative sizes.

\begin{question}\label{quest:lessthan}
    Does $\pi^\top M^{(nb)}_v\pi < \mathcal{K}(G)$ hold for all graphs $G$ for which this is defined?
\end{question}

The examples that we have presented in Section \ref{sec:ex} lead us to conjecture an affirmative answer to this question.  Similar behavior has been observed for the NBRW mixing rate and the SRW mixing rate (see \cite{alon2007non,kempton2016non}) as well as for the edge-space versions of Kemeny's constant \cite{breen2023kemeny}, but in all those cases, there are some exceptions.  We have observed no exceptions to Question \ref{quest:lessthan} in all examples that we have done. 

If the answer to Question \ref{quest:lessthan} is positive, then from there it is of interest to determine how far apart the NBRW and SRW Kemeny's constants can be.  The examples and data that we have presented seem to suggest that in some cases, the non-backtracking condition leads to significant improvement.  It is known, for instance, that the SRW Kemeny's constant for a graph on $n$ vertices is always bounded above by $O(n^3)$, and examples achieving $O(n^3)$ are known \cite{breen2019computing}.  However, no such examples are known for our non-backtracking variants.  

\begin{question}
    Is the NBRW Kemeny's constant for a graph with $n$ vertices always bounded above by $O(n^2)$?
\end{question}

Answers to any of these questions or related questions would significantly improve our understanding of non-backtracking random walks, and how they compare with simple random walks.  

\bibliographystyle{plain}
\bibliography{sources}

\begin{thebibliography}{10}

\bibitem{aleja2019non}
David Aleja, Regino Criado, Alejandro J~Garc{\'\i}a del Amo, {\'A}ngel P{\'e}rez, and Miguel Romance.
\newblock Non-backtracking {P}age{R}ank: From the classic model to {H}ashimoto matrices.
\newblock {\em Chaos, Solitons \& Fractals}, 126:283--291, 2019.

\bibitem{alon2007non}
Noga Alon, Itai Benjamini, Eyal Lubetzky, and Sasha Sodin.
\newblock Non-backtracking random walks mix faster.
\newblock {\em Communications in Contemporary Mathematics}, 9(04):585--603, 2007.

\bibitem{arrigo2019non}
Francesca Arrigo, Desmond~J Higham, and Vanni Noferini.
\newblock Non-backtracking {PageRank}.
\newblock {\em Journal of Scientific Computing}, 80(3):1419--1437, 2019.

\bibitem{barrett2020spanning}
Wayne Barrett, Emily~J Evans, Amanda~E Francis, Mark Kempton, and John Sinkovic.
\newblock Spanning 2-forests and resistance distance in 2-connected graphs.
\newblock {\em Discrete Applied Mathematics}, 284:341--352, 2020.

\bibitem{bendito2008formula}
Enrique Bendito, Angeles Carmona, Andres~M Encinas, and Jose~M Gesto.
\newblock A formula for the {K}irchhoff index.
\newblock {\em International Journal of Quantum Chemistry}, 108(6):1200--1206, 2008.

\bibitem{breen2019computing}
Jane Breen, Steve Butler, Nicklas Day, Colt DeArmond, Kate Lorenzen, Haoyang Qian, and Jacob Riesen.
\newblock Computing {K}emeny's constant for a barbell graph.
\newblock {\em The Electronic Journal of Linear Algebra}, 35:583--598, 2019.

\bibitem{breen2023kemeny}
Jane Breen, Nolan Faught, Cory Glover, Mark Kempton, Adam Knudson, and Alice Oveson.
\newblock Kemeny's constant for nonbacktracking random walks.
\newblock {\em Random Structures \& Algorithms}, 63(2):343--363, 2023.

\bibitem{fasino2023hitting}
Dario Fasino, Arianna Tonetto, and Francesco Tudisco.
\newblock Hitting times for second-order random walks.
\newblock {\em European Journal of Applied Mathematics}, 34(4):642--666, 2023.

\bibitem{faught20221}
Nolan Faught, Mark Kempton, and Adam Knudson.
\newblock A 1-separation formula for the graph {K}emeny constant and {B}raess edges.
\newblock {\em Journal of Mathematical Chemistry}, 60(1):49--69, 2022.

\bibitem{glover2022effects}
Cory Glover, Tyler Jones, Mark Kempton, and Alice Oveson.
\newblock Effects of backtracking on pagerank.
\newblock {\em arXiv preprint arXiv:2211.13353}, 2022.

\bibitem{glover2021some}
Cory Glover and Mark Kempton.
\newblock Some spectral properties of the non-backtracking matrix of a graph.
\newblock {\em Linear Algebra and its Applications}, 618:37--57, 2021.

\bibitem{grinstead1997introduction}
Charles~Miller Grinstead and James~Laurie Snell.
\newblock {\em Introduction to Probability}.
\newblock American Mathematical Soc., 1997.

\bibitem{hunter2006mixing}
Jeffrey~J. Hunter.
\newblock Mixing times with applications to perturbed {M}arkov chains.
\newblock {\em Linear Algebra Appl.}, 417(1):108--123, 2006.

\bibitem{kemenysnell}
John~G. Kemeny and J.~Laurie Snell.
\newblock {\em Finite {M}arkov {C}hains}.
\newblock The University Series in Undergraduate Mathematics. D. Van Nostrand Co., Inc., Princeton, N.J.-Toronto-London-New York, 1960.

\bibitem{kempton2016non}
Mark Kempton.
\newblock Non-backtracking random walks and a weighted {I}hara's theorem.
\newblock {\em Open Journal of Discrete Mathematics}, 6(4):207--226, 2016.

\bibitem{stevesbook}
Stephen~J Kirkland and Michael Neumann.
\newblock {\em Group Inverses of {M}-Matrices and their Applications}.
\newblock CRC Press Boca Raton, FL, 2013.

\bibitem{krzakala2013spectral}
Florent Krzakala, Cristopher Moore, Elchanan Mossel, Joe Neeman, Allan Sly, Lenka Zdeborov{\'a}, and Pan Zhang.
\newblock Spectral redemption in clustering sparse networks.
\newblock {\em Proceedings of the National Academy of Sciences}, 110(52):20935--20940, 2013.

\bibitem{levene2002kemeny}
Mark Levene and George Loizou.
\newblock Kemeny's constant and the random surfer.
\newblock {\em The American Mathematical Monthly}, 109(8):741--745, 2002.

\bibitem{palacios2010bounds}
Jos{\'e}~Luis Palacios and Jos{\'e}~Miguel Renom.
\newblock Bounds for the {K}irchhoff index of regular graphs via the spectra of their random walks.
\newblock {\em International Journal of Quantum Chemistry}, 110(9):1637--1641, 2010.

\bibitem{githubpage}
Matthew Shumway.
\newblock Non-backtracking {K}emeny tools.
\newblock \url{https://github.com/mwshumway/NBRW/tree/main?tab=readme-ov-file}.

\bibitem{wang2017kemeny}
Xiangrong Wang, Johan~LA Dubbeldam, and Piet Van~Mieghem.
\newblock Kemeny's constant and the effective graph resistance.
\newblock {\em Linear Algebra and its Applications}, 535:231--244, 2017.

\end{thebibliography}

\clearpage
\appendix

\section{$C_nC_k$}
\subsection{Proof of Lemma~\ref{lem:CnCk_kemeny}} \label{apx:proof_cnck_kem}
\begin{proof}
    Begin by denoting the two cycles formed by the additional edge as \( C_k \) and \( C_{n+2-k} \). Label the vertices \( \{1, \dots, n+1\} \), starting with one of the degree-three vertices. Continue labeling consecutively around the cycle, first traversing through \( C_k \).
    
    We then consider the following cases.
    \begin{enumerate}
        \item $p, q \in C_k$ with $\deg(p) = \deg(q) = 3$. Either the random walk moves directly to $q$, or it is goes around either $C_k$ or $C_{n+2-k}$ directly to $q$. Thus we compute $\Tilde{m}_{pq} = \frac{1}{3} + \frac{1}{3}(k-1) + \frac{1}{3}(n+1-k) = \frac{1}{3}(n+1)$.
        \item $p, q \in C_k$ with $\deg(p) = \deg(q) = 2$. Here it can be shown that \begin{align*}
            \Tilde{m}_{pq} &= \frac{1}{2}(q-p) + \sum_{m=0}^\infty \left(\frac{1}{2}\right)^{2m+3}\left[p-1+1+m(n-k+2)+(k-q)\right]\\
            &+ 2\sum_{m=1}\left(\frac{1}{2}\right)^{2m+2}[p-1+m(n-k+2)+q-1] \\
            &+ \sum_{m=0}^\infty \left(\frac{1}{2}\right)^{2m+3}[p - 1 + (n-k+1) + m(n-k+2) + (k-q)] \\
            &= \frac{2-k+3n+2q}{6}
        \end{align*}
        \item $p, q \in C_k$ with $\deg(p) = 2, \deg(q) = 3$. Here, the walk will either move directly to $q$ after the first step, or it will end up in the opposite vertex defining the $C_k$, from which it either moves directly to $q$ or does one loop through $C_{n+2-k}$. So we compute $\Tilde{m}_{pq} = \frac{1}{2}(p-1) + \frac{1}{4}(k-p+1) + \frac{1}{4}(k-p+n-k+1) = \frac{1}{4}(n+k)$.
        \item $p, q \in C_k$ with $\deg(p) = 3, \deg(q) = 2$. This can be broken into five different trajectory types, culminating in 
        \begin{align*}
            \Tilde{m}_{pq} &= \frac{1}{3}(q-1) + \sum_{m=0}^\infty \left(\frac{1}{3}\right)\left(\frac{1}{2}\right)^{2m+1}[1 + m(n-k+2) + k-q] \\
            &+ \sum_{m=0}^\infty \left(\frac{1}{3}\right)\left(\frac{1}{2}\right)^{2m+1}[n - k + 1 + m(n-k+2) + k-q] \\
            &+ \sum_{m=0}^\infty \left(\frac{1}{3}\right)\left(\frac{1}{2}\right)^{2m+2}[1+n-k+1 + m(n-k+2) + q-1] \\
            &+ \sum_{m=0}^\infty \left(\frac{1}{3}\right)\left(\frac{1}{2}\right)^{2m+2}[n-k+1+1+m(n-k+2)+q-1] \\
            &= \frac{7-2k+6n+q}{9}
        \end{align*}
        \item $p \in C_k, q \in C_{n+2-k}$ with $\deg(p) = \deg(q) = 2$. Here there are four different trajectories, 
        \begin{align*}
            \Tilde{m}_{pq} &= \sum_{m=0}^\infty \left(\frac{1}{2}\right)^{2m+2}[p-q+mk+q-1] \\
            &+ \sum_{m=0}^\infty \left(\frac{1}{2}\right)^{2m+3}[p+mk+n-k+2-q] \\
            &+ \sum_{m=0}^\infty \left(\frac{1}{2}\right)^{2m+2}[k-p+mk+n-k+2-q] \\
            &+ \sum_{m=0}^\infty \left(\frac{1}{2}\right)^{2m+3}[k-p+1+mk+q-1] \\
            &= \frac{3n+2k+2}{6}
        \end{align*}
        \item $p \in C_{n+2-k}, q \in C_k$ with $\deg(p) = \deg(q) = 2$. This can be calculated by swapping the roles of $k$ and $n+2-k$ in the above argument, thus 
        \begin{align*}
            \Tilde{m}_{pq} = \frac{3n+2(n+2-k) + 2}{6} = \frac{5n-2k+6}{6}
        \end{align*}
        \item $p, q \in C_{n+2-k}$ with $\deg(p) = \deg(q) = 2$. This is very similar to (2), so swap the roles of $k$ with $n+2-k$ to get 
        \[
        \Tilde{m}_{pq} = \frac{2n+k+2q}{6}
        \]
        \item $p \in C_k, q \in C_{n+2-k}$ with $\deg(p) = 3, \deg(q) = 2$. This is the analogue of (3), thus \[
        \Tilde{m}_{pq} = \frac{1}{4}(n + (n+2-k)) = \frac{1}{4}(2n+2-k)
        \]
        \item $p \in C_{n+2-k}, q \in C_k$ with $\deg(p) = 2, \deg(q) = 3$. Similarly, this is the analogue of (4), so we can compute 
        \[
        \Tilde{m}_{pq} = \frac{4n+2k+q+3}{9}
        \]
    \end{enumerate}
    Noticing that $\pi_i = \frac{\deg i}{\text{vol}(G)} = \frac{\deg i}{2(n+1)}$, we compute 
    \begin{align*}
        \pi^\top M_{v}^{nb}\pi &= 2\left[\frac{3}{2(n+1)} \cdot \frac{3}{2(n+1)} \cdot \frac{1}{3}(n+1)\right] \\
        &+ 2\sum_{p=2}^{k-2}\sum_{q=p+1}^{k-1}\left[\frac{2-k+3n+2q}{6} \cdot \frac{2}{2(n+1)} \cdot \frac{2}{2(n+1)}\right] \\
        &+ 2\sum_{p=2}^{k-1}\left[ \frac{2}{2(n+1)} \cdot \frac{3}{2(n+1)} \cdot \frac{1}{4}(n+k) \right] \\
        &+ 2\sum_{q=2}^{k-1}\left[ \frac{2}{2(n+1)} \cdot \frac{3}{2(n+1)} \cdot \frac{7-2k+6n+q}{9} \right] \\
        &+ \sum_{p=2}^{k-1}\sum_{q=2}^{n+1-k}\left[ \frac{2}{2(n+1)} \cdot \frac{2}{2(n+1)} \cdot \frac{3n+2k+2}{6} \right] \\
        &+ \sum_{p=2}^{n+1-k}\sum_{q=2}^{k-1}\left[ \frac{2}{2(n+1)} \cdot \frac{2}{2(n+1)} \cdot \frac{5n-2k+6}{6} \right] \\
        &+ 2\sum_{p=2}^{n-k}\sum_{q=p+1}^{n-k+1} \left[ \frac{2}{2(n+1)} \cdot \frac{2}{2(n+1)} \frac{2n+k+2q}{6} \right] \\
        &+ 2\sum_{p=2}^{n+1-k} \left[ \frac{2}{2(n+1)} \cdot \frac{3}{2(n+1)} \cdot \frac{1}{4}(2n+2-k) \right] \\
        &+ 2\sum_{q=2}^{n+1-k}\left[ \frac{2}{2(n+1)} \cdot \frac{3}{2(n+1)} \cdot \frac{4n+2k+q+3}{9} \right] \\
        &= \frac{10n^3 + n^2(3k+16) + n(-3k^2 + 9k -22) - 3k^2 + 6k-15}{18(n+1)^2}
    \end{align*}
\end{proof}

\section{Cycle Barbell}
\subsection{Proof of Lemma~\ref{lem:cycle_barbell}} \label{apx:proof_cb_kem}
\begin{proof}
Begin by labeling the vertices \( \{1, \dots, a + b + k - 2\} \), starting with a vertex adjacent to the degree-three vertex in \( C_a \). Proceed sequentially by traversing \( C_a \), then \( P_k \), and finally \( C_b \).

To find $M^{(nb)}_v$, we divide the problem into 23 cases. Each case is a different combination of a starting vertex $p$ and target vertex $q$. We compute the mean first passage time in each case. 
\begin{enumerate}
    \item $p,q \in C_a$ with $\deg{p} = \deg{q} = 2$. It can be shown that when $p < q$, $\Tilde{m}_{pq} = \frac{3a}{8} + \frac{b}{2} + \frac{k}{2} + \frac{q+1}{4} - \frac{1}{2}$ and that when $p > q$, $\Tilde{m}_{pq} = \frac{5a}{8} + \frac{b}{2} + \frac{k}{2} - \frac{(q+1)}{4} - \frac{1}{2}$.
    \item $p, q \in C_a$ with $\deg{p} = 2, \deg{q} = 3$. This case is very simple as the random walker is forced directly towards $q$ regardless of the initial choice in step. Thus $\Tilde{m}_{pq} = \frac{a}{2}$. 
    \item $p, q \in C_a$ with $\deg{p} = 3, \deg{q} = 2$. The random walker here may be forced away from q initially if the initial step is taken down the path, so it is slightly more complicated than (2). But it can still be shown that $\Tilde{m}_{pq} = \frac{a}{2} + \frac{2b}{3} + \frac{2k}{3} - \frac{2}{3}$.
    \item $p\in C_a, q\in P_k$ with $\deg{p} = \deg{q} = 2$. Here it is calculated that $\Tilde{m}_{pq} = \frac{a}{2} + q + 1$. 
    \item $p\in P_k, q \in C_a$ with $\deg{p} = \deg{q} = 2$. $\Tilde{m}_{pq} = \frac{a}{2} + b + k -1$. 
    \item $p\in C_a, q\in C_b$ with $\deg{p} = 2, \deg{q} = 3$. Once the random walker escapes from $C_a$, the path towards $q$ is determined. $\Tilde{m}_{pq} = \frac{3a}{2} + k -1$.
    \item $p\in C_b, q\in C_a$ with $\deg{p} = 3, \deg{q} = 2$. Here the random walk may get stuck for a time inside $C_b$, but the path is direct to $q$ after it escapes. $\Tilde{m}_{pq} = \frac{a}{2} = \frac{4b}{3} + k -1$.
    \item $p\in C_a, q \in C_b$ with $\deg{p} = \deg{q} = 2$. The walk may spend time cycling around $C_a$, but move directly towards $q$ afterwards. $\Tilde{m}_{pq} = \frac{3a}{2} + \frac{b}{2} + k - 1$. 
    \item $p\in C_b, q\in C_a$ with $\deg{p} = \deg{q} = 2$. Similar to (8), $\Tilde{m}_{pq} = \frac{3b}{2} + \frac{a}{2} + k-1$.
    \item $p \in C_a, q\in C_b$ with $\deg{p} = \deg{q} = 3$. $\Tilde{m}_{pq} = \frac{4a}{3} + k -1$.
    \item $p\in C_b, q\in C_a$ with $\deg{p} = \deg{q} = 3$. Similar to above, $\Tilde{m}_{pq} = \frac{4b}{3}+k-1$.
    \item $p,q \in P_k$ with $\deg{p} = \deg{q} = 2$. If $p < q$, then $\Tilde{m}_{pq} = q + 1$. Otherwise if $p > q$, then $\Tilde{m}_{pq} = a + b + k - q -2$.
    \item $p\in P_k, q\in C_b$ with $\deg{p} = \deg{q} = 2$. This case (and following cases) are similar to the counterparts when we considered $C_a$ instead of $C_b$. We find $\Tilde{m}_{pq} = a + \frac{b}{2} + k - 1$. 
    \item $p, q \in C_b$ with $\deg{p} = \deg{q} = 2$. If $p < q$ then we find that $\Tilde{m}_{pq} = \frac{a}{4} + \frac{3b}{8} + \frac{k}{4} + \frac{q+1}{4} - \frac{1}{4}$ and if $p > q$, then $\Tilde{m}_{pq} = \frac{3a}{4} + \frac{5b}{8} + \frac{3k}{4} - \frac{(q+1)}{4} - \frac{3}{4}$. 
    \item $p\in C_b, q\in C_a$ with $\deg{p} = 2, \deg{q} = 3$. $\Tilde{m}_{pq} = \frac{3b}{2} + k -1$.
    \item $p\in C_a, q\in C_b$ with $\deg{p} = 3, \deg{q} = 2$. $\Tilde{m}_{pq} = \frac{4a}{3} + \frac{b}{2} + k-1$.
    \item $p\in C_b, q\in C_b$ with $\deg{p}=2, \deg{q}=3$. $\Tilde{m}_{pq} = \frac{b}{2}$.
    \item $p\in C_b, q\in C_b$ with $\deg{p} = 3, \deg{q} = 2$. $\Tilde{m}_{pq} = \frac{2a}{3} + \frac{b}{2} + \frac{2k}{3} - \frac{2}{3}$.
    \item $p\in C_b, q\in C_k$ with $\deg{p} = \deg{q} = 2$. $\Tilde{m}_{pq} = a + \frac{3b}{2} + k - q -2$.
    \item $p\in C_a, q\in P_k$ with $\deg{p} = 3, \deg{q}=2$. $\Tilde{m}_{pq} = \frac{a}{3} + q + 1$. 
    \item $p\in C_b, q\in P_k$ with $\deg{p}=\deg{q}=2$. $\Tilde{m}_{pq} = a + \frac{4b}{3} + k -q -1$.
    \item $p\in P_k, q\in C_a$ with $\deg{p}=2, \deg{q} = 3$. $\Tilde{m}_{pq} = b + k -1$.
    \item $p\in P_k, q\in C_b$ with $\deg{p}=2, \deg{q}=3$. $\Tilde{m}_{pq} = a + k -1$.
\end{enumerate}

Recall that $|E(CB)| = a + b + k - 1$. This enables us to determine the entries of the stationary vector, $\pi$. Almost every entry in $\pi$ will be $\frac{2}{2(a + b + k -1)}$; there are only 2 entries that are different, corresponding the 2 vertices of degree 3: $\frac{3}{2(a+b+k-1)}$. Defining $M = 2(a+b+k-1)$, we can compute Kemeny's constant as follows. 

\begin{align*}
    \pi^TM^{(nb)}_v\pi &= \frac{2}{M} \left[
  \frac{2}{M}\sum_{p=1}^{a-2}\sum_{q=p+1}^{a-1} \left(\frac{3a}{8} + \frac{b}{2} + \frac{k}{2} + \frac{q}{4} - \frac{1}{2}\right)
  + \frac{2}{M}\sum_{p=2}^{a-1} \sum_{q=1}^{p-1} \left(\frac{5a}{8} + \frac{b}{2} + \frac{k}{2} - \frac{q}{4} - \frac{1}{2}\right) \right. \\
  &\quad \left. + \frac{3}{M}\sum_{p=1}^{a-1} \left(\frac{a}{2}\right)
  + \sum_{p=1}^{a-1} \frac{2}{M}\sum_{q=a+1}^{a+k-2} \left(\frac{a}{2} + q\right) \right. \\
  &\quad \left. + \frac{3}{M}\sum_{p=1}^{a-1} \left(\frac{3a}{2} + k - 1\right)
  + \frac{2}{M}\sum_{p=1}^{a-1} \sum_{q=a+k}^{a+b+k-2} \left(\frac{3a}{2} + \frac{b}{2} + k - 1\right)
\right] \\
&+ \frac{3}{M} \left[
  \frac{2}{M}\sum_{q=1}^{a-1}\left(\frac{a}{2}+\frac{2b}{3}+\frac{2k}{3}-\frac{2}{3}\right)
  + \frac{3}{M}\left(\frac{4a}{3}+k-1\right) \right. \\
&\quad \left. + \frac{2}{M}\sum_{q=a+k}^{a+b+k-2}\left(\frac{4a}{3} + \frac{b}{2} + k - 1\right)
  + \frac{2}{M}\sum_{q=a+1}^{a+k-2}\left(\frac{a}{3} + q\right)
\right] \\
&+ \frac{2}{M} \left[
  \frac{2}{M}\sum_{p=a+1}^{a+k-2} \sum_{q=1}^{a-1} \left(\frac{a}{2} + b + k - 1\right)
  + \frac{2}{M}\sum_{p=a+1}^{a+k-2} \sum_{q=a+k}^{a+b+k-2} \left(a + \frac{b}{2} + k - 1\right) \right. \\
  &\quad \left. + \frac{3}{M}\sum_{p=a+1}^{a+k-2} \left(b + k - 1\right)
  + \frac{3}{M}\sum_{p=a+1}^{a+k-2} \left(a + k - 1\right) \right. \\
  &\quad + \left. \frac{2}{M}\sum_{p=a+1}^{a+k-2} \sum_{q=a+1}^{p-1} \left(a + b + k - q - 1\right)
  + \frac{2}{M}\sum_{p=a+1}^{a+k-2} \sum_{q=p+1}^{a+k-2} \left(q\right)
\right] \\
&+ \frac{3}{M} \left[
    \frac{2}{M}\sum_{q=1}^{a-1}\left(\frac{a}{2} + \frac{4b}{3} + k - 1\right)
    + \frac{3}{M}\left(\frac{4b}{3} + k -1\right) \right. \\
    &\quad + \left. \frac{2}{M}\sum_{q=a+k}^{a+b+k-2}\left(\frac{2a}{3} + \frac{b}{2} + \frac{2k}{3} - \frac{2}{3}\right)
    + \frac{2}{M}\sum_{q=a+1}^{a+k-2}\left(a + \frac{4b}{3} + k - q -1\right)
\right] \\
&+ \frac{2}{M} \left[ 
\frac{2}{M}\sum_{p=a+k}^{a+b+k-2} \sum_{q=1}^{a-1} \left( \frac{3b}{2} + \frac{a}{2} + k - 1 \right) 
+ \frac{2}{M}\sum_{p=a+k}^{a+b+k-3} \sum_{q=p+1}^{a+b+k-2} \left( \frac{a}{4} + \frac{3b}{8} + \frac{k}{4} + \frac{q}{4} - \frac{1}{4} \right) \right. \\
&\quad + \left. \frac{2}{M}\sum_{p=a+k}^{a+b+k-2} \sum_{q=a+k}^{p-1} \left( \frac{3a}{4} + \frac{5b}{8} + \frac{3k}{4} - \frac{q}{4} - \frac{3}{4} \right) 
+ \frac{3}{M}\sum_{p=a+k}^{a+b+k-2} \left( \frac{3b}{2} + k - 1 \right) \right. \\
&\quad + \left. \frac{3}{M}\sum_{p=a+k}^{a+b+k-2} \left( \frac{b}{2} \right)
+` \frac{2}{M}\sum_{p=a+k}^{a+b+k-2} \sum_{q=a+1}^{a+k-2} \left( a + \frac{3b}{2} + k - q - 1 \right)
\right]
 \end{align*}
The terms in brackets then condense to 
\begin{align*}
&=\frac{(-1 + a) (13a^2 + a (-74 + 48b + 48k) + 12 (3 + b^2 - 4k + k^2 + b (-5 + 2k)))}{6 (2 (a + b + k - 1))^2} \\
&+ \frac{7 + 3a^2 + 3b^2 - 10k + 3k^2 + b (-13 + 6k) + a (-19 + 12b + 12k)}{(2 (a + b + k - 1))^2} \\
&+ \frac{2 (-2 + k) (6 + 3a^2 + 3b^2 - 10k + 4k^2 + 3b (-5 + 3k) + 3a (-5 + 4b + 3k))}{3 (2 (a + b + k - 1))^2} \\
&+ \frac{7 + 3a^2 + 3b^2 - 10k + 3k^2 + a (-13 + 12b + 6k) + b (-19 + 12k)}{4 (-1 + a + b + k)^2} \\
&+ \frac{(-1 + b) (12a^2 + 13b^2 + 12a (-5 + 4b + 2k) + b (-74 + 48k) + 12 (3 - 4k + k^2))}{6 (2 (a + b + k - 1))^2} \\
\end{align*}
Finally the above expression simplifies to the desired expression, \begin{align*}
    &\frac{13a^3 + 13b^3 + b^2(-87 + 60k) + a^2(-87 + 60b + 60k) + 2b(49 - 72k + 24k^2)}{24(-1 + a + b + k)^2} \\
&+ \frac{4(-9 + 20k - 15k^2 + 4k^3) + 2a(49 + 30b^2 - 72k + 24k^2 + 12b(-7 + 4k))}{24(-1 + a + b + k)^2}
\end{align*}
\end{proof}

\subsection{Proof of Lemma~\ref{lem:cb_max}} \label{apx:proof_cb_max}
\begin{proof}
    Denote the expression of the non-backtracking Kemeny's constant of $CB(k, a, b)$ (the right hand side of the equation in Lemma \ref{lem:cycle_barbell}) by 
    \[
    f(k, a, b) = \frac{N(k, a, b)}{24(-1 + a + b + k)^2}
    \]
    where $N(k, a, b)$ is the function in the numerator. Fix $n$ and recalling that $n=a + b + k -2$, note that the denominator of $f$, $24(n+1)^2$ is also fixed. Therefore, we restrict our focus to maximizing the numerator $N(k, a, b)$. Now fix $k$. Thus the expression $a + b = n - k + 2$ is also fixed. Define the fixed quantity $s := n-k+2$ so that we can write $b = s - a$. Consider then $N(k, a, s-a)$, which simplifies to 
    \[
    N(k, a, s-a) = \alpha a^2 + \beta a + \gamma
    \]
    where 
    \begin{align*}
        \alpha &= 3(-7n + 15k - 16) \\
        \beta &= 3(7n^2 + 30n + 15k^2 - 2k(23 + 11n) + 32)
    \end{align*}
    and $\gamma$ is a quantity that depends only on $n$ and $k$. This shows that $N(k, a, s-a)$ is a quadratic function of a single variable, $a$, and thus has critical point 
    \[
    a^* = -\frac{\beta}{2\alpha} = \frac{n-k+2}{2}
    \]
    At this point, note that the feasible regions for parameters $k, a, b$ are as follows:
    \begin{align*}
        3 \leq a, b \leq n-k-1 \\
        2 \leq k \leq n - 4  \\
        6 \leq n
    \end{align*}
    which are purely informed by the structure of the cycle barbell graph. 

    Now, $a^*$ is indeed the maximizer of $N(k, a, s-a)$ whenever $\alpha < 0$ as the function would then be concave. If $\alpha > 0$, then the function is convex and the maximizer would be found at the end points of the feasible region for $a$, i.e. $a \in \{3, n-k-1\}$. If $\alpha=0$, then we have a linear function in $a$ and will also have a maximizer at the end points of the feasible region. It will thus be considered with the case that $\alpha > 0$.

    The proof follows by showing in cases that when $\alpha < 0$, that $N\left(k, \frac{n-k+2}{2}, \frac{n-k+2}{2}\right)$ is a non-increasing function of $k$, thus maximized at $N\left(2, \frac{n}{2}, \frac{n}{2}\right)$. For simplicity of argument, we assume that $n$ is even such that $\frac{n}{2}$ is an integer and assert that the case of odd $n$ follows similarly. In the case that $\alpha > 0$, we assume without loss of generality (as $N(k, a, b)$ is symmetric in $a$ and $b$) that the optimal $a$ is $3$ and that $b$ is $n-k-1$, then show that
    \[
    \underset{2\leq k \leq n-4}{\max}N(k, 3, n-k-1) \leq N(2, n/2, n/2)
    \]
    which proves the result. \\

    \noindent{Case 1: $\alpha < 0$.}
    Observe, 

    \begin{align*}
        N\left(k, \frac{n-k+2}{2}, \frac{n-k+2}{2}\right) - N\left(k+1, \frac{n-k+1}{2}, \frac{n-k+1}{2}\right) = -\frac{3}{4}\left(33 + 15 k^2 + 25 n - n^2 - k (17 + 14 n)\right)
    \end{align*}

    Define $Q(k) := 15k^2 - k(17+14n) - n^2 + 25n +33$. We will show that for all $n \geq 7$, $Q(k) < 0$ for all $k$. Note that $Q(k)$ is a convex quadratic function in $k$, so it will achieve it's maximum at its endpoints, either $k=2$ or $k=n-4$. We see that 
    \[
    Q(2) = 59 - 3n - n^2
    \]
    which is negative for all integers $n \geq 7$. Additionally, 
    \[
    Q(n-4) = 341 - 56n
    \]
    which is also negative for all integers $n \geq 7$. Thus for $n \geq 7$, $N\left(k, \frac{n-k+2}{2}, \frac{n-k+2}{2}\right)$ is non-increasing as a function of $k$ and is thus maximized at $k=2$ and $CB(k=2, a=n/2, b=n/2)$.
    At $n=6$, there is only one feasible graph: $CB(k=2, a=3, b=3)$, which aligns with the conclusion of the lemma. \\

    \noindent{Case 2: $\alpha \geq 0$.}
    In this setting, the quadratic in $a$ is convex, so the maximizer lies on the boundary.  
    By symmetry in $a,b$, we may take $a^* = 3$ and $b^* = n-k-1$.  
    To compare with the balanced candidate $(k,a,b) = (2,n/2,n/2)$, consider the difference
    \[
    D(k,n) := N\!\left(2,\tfrac{n}{2},\tfrac{n}{2}\right) - N\!\left(k,3,n-k-1\right)
    = \tfrac{3}{4}\Bigl(320 - 20k^{3} + 44k^{2}(n-2) + 20n - 42n^{2} + 7n^{3} - 4k(39 - 28n + 7n^{2})\Bigr).
    \]

    Direct computation verifies that $D(k, n) \ge 0$ for all integers $k \ge 2$ and $n \ge 6$, which are precisely the conditions defining our feasible set.
    

    Therefore, regardless of the maximizing pair $(a^*,b^*)$ for fixed $k$, the global maximum of $N(k,a,b)$ at fixed $n$ occurs when $k=2$ and $a,b$ are as balanced as possible.

\end{proof}

\end{document}